\RequirePackage{rotating}

\documentclass{amsart}  

\usepackage[foot]{amsaddr}

\usepackage[usenames, dvipsnames]{color}
\usepackage{amsfonts}
\usepackage{amsthm}
\usepackage{amsmath}
\usepackage{amsfonts}
\usepackage{latexsym}
\usepackage{amssymb}
\usepackage{amscd}
\usepackage[latin1]{inputenc}
\usepackage{verbatim}
\usepackage{enumerate}
\usepackage{enumitem}
\usepackage{graphicx}
\usepackage{xcolor}
\usepackage{tikz}
\usepackage{caption}
\usepackage{adjustbox}
\usepackage{algorithmic}
\usepackage{epstopdf}
\usepackage{fancybox}
\usepackage{booktabs} 
\usepackage{expl3}
\usepackage{l3keys2e}
\usepackage{xparse}
\usepackage{blkarray}
\usepackage{pst-node}
\usepackage{animate}
\usepackage{float}
\usepackage{booktabs}
\usepackage{makecell}
\usepackage{systeme}
\usepackage{placeins}
\usetikzlibrary{matrix}
\usetikzlibrary {positioning}
\usetikzlibrary{arrows,shapes}
\usetikzlibrary{trees}
\usetikzlibrary{shapes.geometric}
\usetikzlibrary{calc,shapes.callouts,shapes.arrows}
\usetikzlibrary{graphs}
\usetikzlibrary{positioning, arrows}

\definecolor{mybluegreen}{rgb}{0.1, 0.55, 0.35}
\usepackage{caption}
\usepackage{array,booktabs}
\usepackage{siunitx}
\usepackage{rotating}

\usepackage[spaces,hyphens]{url}
\usepackage[colorlinks,allcolors=blue]{hyperref}
\usepackage[math]{cellspace}
\newtheorem{theorem}{Theorem}[section]
\newtheorem{lemma}[theorem]{Lemma}

\newtheorem{example}[theorem]{Example}

\newtheorem{definition}[theorem]{Definition}
\newtheorem{conj}{Conjecture}[section]

\theoremstyle{definition}

\newcommand{\Sym}{\mathrm{Sym}} 
\newcommand{\ind}{\mathrm{ind}}
\DeclareMathOperator{\spanof}{span}

\definecolor{mybluegreen}{rgb}{0.1, 0.55, 0.35}
\newcolumntype{?}{!{\vrule width 2pt}}
\newcolumntype{M}[1]{>{\centering\arraybackslash}m{#1}}
\newcolumntype{P}[1]{>{\centering\arraybackslash}p{#1}}
\newcolumntype{N}{@{}m{0pt}@{}}
\title[EKR Theorem for $2$-intersecting families of perfect matchings]{The Erd\H{o}s-Ko-Rado theorem for $2$-intersecting families of perfect matchings}

\author[S.~Fallat]{Shaun Fallat${^1}$} \email[S.~Fallat]{shaun.fallat@uregina.ca}
\author[K.~Meagher]{Karen Meagher${^2}$ ${^*}$} \email[K.~Meagher]{karen.meagher@uregina.ca}
\author[M.~N.~Shirazi]{Mahsa N. Shirazi} \email[M.~N.~Shirazi]{mahsa.nasrollahi@gmail.com}
\address{Department of Mathematics and Statistics, University of Regina, Regina, SK, S4S 0A2,
Canada}

\thanks{${^1}$Research supported in part by an NSERC Discovery Research Grant,
    Application No.: RGPIN--2019--03934.}
\thanks{${^2}$Research supported in part by an NSERC Discovery Research Grant,
    Application No.: RGPIN-03852-2018.}
    \thanks{${^*}$Corresponding Author}

\date{\today}

 		
\begin{document}

\begin{abstract}  
A \textsl{perfect matching} in the complete graph on $2k$ vertices is a set of edges such that no two edges have a vertex in common and every vertex is covered exactly once. Two perfect matchings are said to be $t$-intersecting if they have at least $t$ edges in common. The main result in this paper is an extension of the famous Erd\H{o}s-Ko-Rado (EKR) theorem~\cite{EKR} to 2-intersecting families of perfect matchings for all values of $k$. Specifically, for $k\geq 3$ a set of 2-intersecting perfect matchings in $K_{2k}$ of maximum size has $(2k-5)(2k-7)\cdots (1)$ perfect matchings.  \\
\end{abstract}

\maketitle

\section{Introduction and Preliminaries}

In this paper we present two different approaches to establish a version of the Erd\H{o}s-Ko-Rado theorem for $2$-intersecting families of perfect matchings. There are many recent results that verify analogs of the  Erd\H{o}s-Ko-Rado theorem. This research area started with Erd\H{o}s, Ko, and Rado's work on systems of intersecting sets. In 1961, they proved if $\mathcal{F}$ is a $t$-intersecting family of $k$-subsets of $\{1,2,\ldots, n\}$, then there is a tight upper bound on the size of $\mathcal{F}$ with $n$ sufficiently large~\cite{EKR}.
\begin{theorem}[EKR]\cite{EKR}
If $\mathcal{F}$ is a $t$-intersecting family of $k$-subsets of $\{1,2,\ldots, n\}$, then there exists a function $f(k,t)$ such that if $n\geq f(k,t)$, then 
\begin{equation*}
|\mathcal{F}|\leq  \binom{n-t}{k-t}.
\end{equation*}
If equality holds, then $\mathcal{F}$ consists of all $k$-subsets containing a fixed $t$-subset of $\{1,2,\ldots, n\}$.

\end{theorem}
Twenty-three years after the publication of Erd\H{o}s, Ko and Rado's work, Wilson~\cite{W} enhanced their results by giving an algebraic proof of the their result with the exact value of $f(k, t)$ for all $k$ and $t$. Later in 1997, Ahlswede and Khachatrian~\cite{AK} found all maximum $t$-intersecting families of $k$-subsets for all values of $n$. In 2011, Ellis, Friedgut, and Pilpel~\cite{Ellis} showed that the analog of the EKR theorem holds for $t$-intersecting families of permutations of $\{1,\dots, n \}$, when $n$ is sufficiently large relative to $t$.  In 2005, Meagher and Moura~\cite{MM} proved that a natural version of the EKR theorem holds for uniform set-partitions. Recently, an algebraic proof of this well-known theorem for intersecting families of perfect matching was found by Godsil and Meagher~\cite{GMP} and their proof is based on eigenvalue techniques originally utilized by Wilson~\cite{W}. Further, they conjectured a version of the EKR theorem holds for $t$-intersecting families of perfect matchings, when $2k\geq 3t+2$. In 2018, Lindzey~\cite{L} proved this conjecture for all $t$, provided that $k$ is sufficiently large relative to $t$. In this paper we prove the conjecture holds for $t=2$ and all $k\geq 3$. \\

In Section~\ref{sec:background}, we provide some necessary background on perfect matchings and introduce the association scheme for perfect matchings. We convert the problem of finding the maximum size of an intersecting set of perfect matchings to the problem of finding a maximum coclique in a graph. Section~\ref{sec:FirstApproach} gives a proof of the result for some values of $k$; this proof uses the well-known clique-coclique bound. In Section~\ref{sec:Second Approach}, we derive a different approach that proves the result for all $k$. In this section we construct a matrix in the association scheme that is a weighted adjacency matrix for the graph in question, we  prove our result by showing the ratio bound holds with equality for this weighted adjacency matrix. We conclude this work with some possible related future directions and open problems and we include an appendix which provides several partial tables of eigenvalues for different graphs in the association scheme for perfect matchings.

\section{Background on Perfect Matchings}
\label{sec:background}

A \textsl{matching} $M$ in a graph $X$ is a set of edges such that no two edges have a vertex in common. If a matching covers every vertex of $X$, it is called a \textsl{perfect matching}~\cite{GG}. Two perfect matchings are said to be \textsl{$t$-intersecting} if they have at least $t$ edges in common. If $t=1$, we just say that they are intersecting. In this paper we only consider perfect matchings in complete graphs with an even number of vertices. Our goal is to find the size of the largest set of $2$-intersecting perfect matchings in $K_{2k}$ for all $k \geq 3$. A perfect matching is a special case of a uniform set-partition in which the size of each part is 2. In~\cite{MM} a proof for a version of the EKR theorem for uniform set-partitions is presented. However, the results given in~\cite{MM} are asymptotic in nature (as in the size of $k$ needs to be sufficiently large relative to $t$) and do not apply to perfect matchings when $t > 1$.

It is easy to check that the number of perfect matchings in $K_{2k}$ is
\begin{equation*}
\frac{1}{k!}\binom{2k}{2}\binom{2k-2}{2}\cdots \binom{2}{2} = (2k-1)(2k-3)(2k-5)\cdots 1.
\end{equation*}
For any positive integer $k$ define 
\begin{equation*}
(2k-1)!!:=(2k-1)(2k-3)(2k-5)\cdots 1,
\end{equation*}
so the number of perfect matchings in $K_{2k}$ is $(2k-1)!!$.

A set of all perfect matchings that contain a common set of $t$ edges is called a \textsl{canonically $t$-intersecting set}. The size of a canonically $t$-intersecting set of perfect matchings in $K_{2k}$ is $(2k-2t-1)!!$. For a set $T$ of $t$ disjoint edges in $K_{2k}$, we use $\nu_{T}$ to denote the characteristic vector of the set of all perfect matchings that include all the edges in $T$.

\subsection{Perfect matching derangement graph}

The approach we take is to define a graph in which every coclique is a set of intersecting perfect matchings. We then use algebraic techniques to find the size of the largest cocliques in this graph. To start, we state some well-known terminology.

Let $X$ be a graph. A \textsl{clique} in $X$ is a set of vertices in which any two are adjacent; a \textsl{coclique} is a set of vertices in which no two are adjacent. The size of a largest clique and a largest coclique are denoted by $\omega(X)$ and $\alpha(X)$, respectively. The \textsl{adjacency matrix} $A(X)$ of $X$ is a matrix in which rows and columns are indexed by the vertices and the $(i ,j)$-entry is 1 if $i\sim j$, and 0 otherwise. A \textsl{weighted adjacency matrix} $A_{W}(X)$ subordinate to $X$ is a symmetric matrix in which rows and columns are indexed by the vertices and the $(i, j)$-entry may be non-zero (which is interpreted as its edge weight) if $i\sim j$ and is 0 otherwise. The \textsl{eigenvalues} of $X$ refer to the eigenvalues of its adjacency matrix. We  use $\mathbf{1}$ to denote the all-ones vector; for any $d$-regular graph, the all-ones vector is an eigenvector with eigenvalue $d$.
 
In general, finding the largest coclique of a graph $X$ is a well-known NP-hard problem, but there is a famous upper bound on $\alpha(X)$ that we  use throughout this paper.

\begin{theorem}[Delsarte-Hoffman bound]\cite[p. 31]{GMB}\label{ratioBound}
Let $A$ be a weighted adjacency matrix for a graph $X$ on vertex set $V(X)$. If $A$ has constant row sum $d$ and least eigenvalue $\tau$, then 
\begin{equation*}\label{RatioBound}
\alpha(X)\leq \frac{|V(X)|}{1-\frac{d}{\tau}}.
\end{equation*}
If equality holds for some coclique $S$ with characteristic vector $\nu_{S}$, then 
\begin{equation*}
\nu_{S}-\frac{|S|}{|V(X)|}\mathbf{1}
\end{equation*}
is an eigenvector with eigenvalue $\tau$.
\end{theorem}

This bound is based on the ratio between the largest and the smallest eigenvalue for a weighted adjacency matrix, thus it is also known as the \textsl{Ratio Bound}. The Ratio Bound is important here since we  apply it to a graph defined so that the cocliques are sets of $2$-intersecting perfect matchings.

\begin{definition}\label{Mt(2k)}\cite{GMP}
Define the \textsl{perfect matching derangement graph} $M_{t}(2k)$ to be the graph whose vertices are perfect matchings on complete graph $K_{2k}$. In this graph two vertices are adjacent if they have at most $(t-1)$ edges in common. Denote the adjacency matrix of $M_{t}(2k)$ by $A_{t}(2k)$.
\end{definition}

In a coclique of $M_{t}(2k)$, any two vertices are not adjacent; thus they have more than $t-1$ edges in common or in other words, they are $t$-intersecting perfect matchings. Using the Delsarte-Hoffman bound, our problem transforms into finding a weighted adjacency matrix for $M_2(2k)$, for any $k \geq 3$, with a sufficiently large ratio between the largest and least eigenvalues. This method to prove EKR theorems was first developed by Wilson in 1984~\cite{W}.  In 2015, Godsil and Meagher applied this method to the family of all perfect matchings of the complete graph $K_{2k}$ to find the largest set of intersecting perfect matchings ($t=1$)~\cite{GMP}; later in 2017 it was applied to $t$-intersecting perfect matchings by Lindzey~\cite{L}. \\

\begin{example}[$M_{1}(6)$]
In Definition~\ref{Mt(2k)}, let $t=1$ and $2k=6$. The number of perfect matchings in $K_{6}$ is $5!!$, so $M_{1}(6)$ has 15 vertices. Two vertices here are adjacent if they are not intersecting. Therefore we have,
\begin{figure}[H]
\includegraphics[scale=0.5]{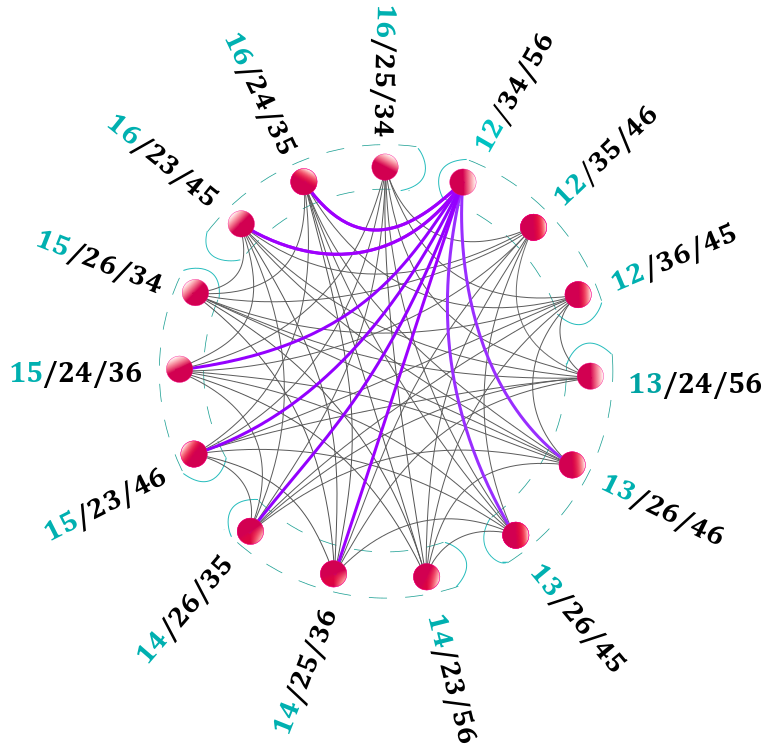}
\caption{Graph $M_{t}(2k)$ when $t=1$ and $2k=6$}
\end{figure}
\end{example}

\subsection{Perfect matching association scheme} 

A \textit{set partition} $P=\{P_1, \ldots, P_\ell\}$, of the set $\{1,2,\ldots, n\}$ is a grouping the elements of this set into nonempty subsets (parts) such that each element is included in exactly one part. A set partition in which all parts have even size is called an \textit{even set partition}. 
 An \textsl{integer partition} of a positive integer $n$ is a list $\lambda = [\lambda_1,\lambda_2, \dots, \lambda_\ell]$ of positive integers with $n = \sum_i \lambda_i$; typically in an integer partition $\lambda_i \geq \lambda_{i+1}$. We  denote an integer partition of $n$ by $\lambda \vdash n$. An integer partition in which all parts have even size is called an \textsl{even partition}. For example, if $\lambda \vdash k$, with $\lambda = [\lambda_1,\lambda_2, \dots, \lambda_\ell]$, then $2\lambda = [ 2\lambda_1, 2\lambda_2, \dots, 2\lambda_\ell]$ is an even partition. The set of sizes of the parts in a set partition $P$ of $\{1,2\dots,n\}$ is an integer partition of $n$, we  call the integer partition the \textsl{shape} of $P$. (See~\cite[Sections 3.1 and 15.4]{GMB} for more details about partitions.)\\

It is easily seen that any perfect matching of $K_{2k}$ is an even set partition with all parts of size 2; and the shape of any perfect matching is the integer partition $[2,2,\dots,2]$. Taking the union of (or overlapping) two perfect matchings in $K_{2k}$ produces disjoint even cycles in $K_{2k}$, where the union of parallel edges gives rise to 2-cycles. Any two perfect matchings in $K_{2k}$ produce a set partition of the set $\{1,\dots,2k\}$ where each part is the set of vertices contained in one of the even cycles in the union of the two perfect matchings. The shape of this set partition is the integer partition $\lambda = [ \lambda_{1},\lambda_{2}, \cdots, \lambda_{\ell} ]$, here the cycles have size $\lambda_{i}$. Such an integer partition will be even. Using this we  define a set of graphs on the collection of all perfect matchings.

\begin{definition}\label{CycleType}
Let $k$ be an integer and $\lambda \vdash k$. Define $A_{2\lambda}$ to be a matrix in which rows and columns are indexed by perfect matchings of $K_{2k}$. In $A_{2\lambda}$ the $(P,Q)$-entry is 1 if the union of the perfect matchings $P$ and $Q$ has shape $2\lambda$, and 0 otherwise.
\end{definition}

We define 
\[
\mathcal{A}=\{ A_{2\lambda} \, | \, \lambda \vdash k\}.
\] 
The set $\mathcal{A}$ forms a symmetric association scheme which is known as the \textsl{perfect matching association scheme}~\cite[Section 15.4]{GMB}, and the matrix algebra $\mathbb{C}[\mathcal{A}]$ constructed by the complex linear combinations of the matrices in $\mathcal{A}$ is called the \textsl{Bose-Mesner algebra} of this association scheme. For a matrix $A_{2\lambda} \in \mathcal{A}$ define the graph $X_{2\lambda}$ so that $A_{2\lambda}$ is its adjacency matrix. A \textsl{graph in the association scheme} is any graph $X$ with $A(X) \in \mathbb{C}[\mathcal{A}]$. Every graph in this association scheme is an undirected graph with the set of perfect matchings as its vertex set. 

The symmetric group $\Sym(2k)$ acts transitively on the set of perfect matchings of $K_{2k}$. This action preserves the adjacencies 
in  $M_{t}(2k)$. This implies $\Sym(2k)$ is a subset of the automorphism group of the graph $M_{t}(2k)$ and that $M_{t}(2k)$ is vertex transitive. For any $\lambda \vdash k$, consider the set of all pairs of perfect matchings $P$ and $Q$, with the
property that the union of $P$ and $Q$ has shape $2\lambda$. It is known that $\Sym(2k)$ is transitive on this set of pairs (again see~\cite[Section 15.4]{GMB}), this implies each graph $X_{2\lambda}$ is edge transitive.

We state some properties of this association scheme and refer the reader to~\cite[Chapter 3]{GMB} for more details and proofs. Denote the \textsl{Schur product} of two matrices $A=[a_{i,j}]$ and $B = [b_{i,j}]$ by the matrix $A \circ B $ whose $(i,j)$-entry is given by $[a_{i,j}b_{i,j}]$. For any two matrices $A_{\lambda_{i}}$ and $A_{\lambda_{j}}$ in $\mathcal{A}$,
\begin{equation*}
A_{\lambda_{i}}\circ A_{\lambda_{j}} = \left\{
\begin{array}{ll}
      0 & i\neq i , \\
      A_{\lambda_{i}} & i=j.
\end{array}
\right.
\end{equation*}
This implies the matrices $A_{\lambda_{i}}$ are linearly independent and Schur orthogonal. Hence the matrix set $\{ A_{2\lambda} \, | \, \lambda \vdash n \}$ is an orthogonal basis for the Bose-Mesner algebra of this scheme.
Further, the matrices in $\mathcal{A}$ are symmetric and commute, thus they are simultaneously diagonalizable. 

The group $\Sym(2k)$ acts transitively on the set of perfect matchings of $K_{2k}$, and the stabilizer of a single perfect matching under this action is isomorphic to the wreath product $\Sym(2) \wr \Sym(k)$. The action of the group $\Sym(2k)$ on the set of perfect matchings is equivalent its action on the cosets of $\Sym(2k) / \left( \Sym(2) \wr \Sym(k) \right)$. The perfect matching scheme is the Schurian association scheme, or  the orbital scheme of the action of $\Sym(2k)$ on the cosets $\Sym(2k) / \left( \Sym(2) \wr \Sym(k) \right)$. 

The $(2k-1)!!$-dimensional vector space of vectors indexed by the perfect matchings is a 
$\Sym(2k)$-module. It is well-known that the irreducible representations of $\Sym(2k)$ correspond to integer partitions of $2k$~\cite{Ra} and that this module can be expressed as the sum of irreducible modules $\Sym(2k)$ corresponding to  even integer partitions of $2k$. We  denote these irreducible modules by the corresponding even integer partitions (see~\cite[Chapter 15]{GMB} for details). (For example, $[2k]$ will be used to denote the irreducible module corresponding to the trivial representation; this is the 1-dimensional vector space of constant vectors of length $(2k-1)!!$.) It follows that the common eigenspaces of the matrices in the perfect matching association scheme are unions of these irreducible modules, thus the common
eigenspaces in the perfect matching association scheme correspond to the even integer partitions of $2k$. Further details on these eigenspaces are in Subsection~\ref{Mod of Char Table}.

The graph $M_{t}(2k)$ is the union of the graphs of the scheme $X_{\lambda}$ in which the even partition $\lambda$ has at most $t-1$ cycles of length 2, this can be expressed as
\begin{equation*}
A_{t}(2k) = \sum_{\lambda \vdash 2k}A_{\lambda},
\end{equation*}
where the sum is take over partitions $\lambda$ have at most $t-1$ parts of length 2. Further, the eigenvalues of $M_{t}(2k)$ are the sums of the eigenvalues of matrices $A_{\lambda}$ in this sum.

\section{Clique-Coclique Approach}
\label{sec:FirstApproach}

For the first approach, we  construct a large clique in $M_2(2k)$. This clique is constructed using \textsl{difference sets} and \textsl{projective planes} which are introduced and discussed in the next subsection. Then we  use this clique to construct a weighted adjacency matrix $M$ of $M_{2}(2k)$, and we will show that the Delsarte-Hoffman bound holds with equality for this matrix. This approach only works for some values of $k$, but motivates the second approach.

We follow the approach from~\cite{MR3990672}, where the authors show how to find the matrix used in Wilson's proof of the EKR Theorem~\cite{W}.

\subsection{Singer difference sets and finite projective plane }
\label{sec:SinderDS}

A \textsl{symmetric balanced incomplete design} with parameters $(v, k, \lambda)$ is a collection of subsets from a base set of size $v$; these subsets are called blocks; with the property that each block has exactly $k\geq 1$ elements where $v>k$, each element appears in exactly $k$ blocks, and each pair of elements appears in exactly $\lambda \geq 1$ blocks. 

\begin{definition}\label{FPP}\cite[p.51]{Wallis}
A symmetric balanced incomplete block design with parameters $(n^{2}+n+1, n+1, 1)$ is called a \textsl{finite projective plane of order $n$} in which the base set is the point set of size $n^{2}+n+1$ and each block ($(n+1)$-subset) in this design a block is called a line.

Alternatively, A finite projective plane consists a finite set of points $P$ and a set $L$ of subsets of $P$, called lines, satisfying the axioms (P1), (P2), and (P3):
\begin{itemize}
\item[(P1)]Given two points in $P$, there is exactly one line that contains both.

\item[(P2)]Given two lines in $L$, there is exactly one point on both.

\item[(P3)]There are four points of which no three are co-linear.
\end{itemize}
\end{definition}

A difference set $(v, k, \lambda)$-DS is a $k$-set $B$ of an abelian group $G$ of order $v$ in which for any nonzero element $g$ of $G$ there are exactly $\lambda$ ordered pairs in $B$ such that $g$ can be constructed as their difference~\cite[Section 5.1]{Wallis}. If the group $G$ is cyclic, then the difference set $(v, k, \lambda)$-DS is called a cyclic difference set. 

A difference set $D = \{d_1 , d_2, \dots , d_k \}$ can be \textsl{developed} in to a symmetric balanced incomplete block design with the blocks
\[
D + g := \{d_1+g, d_2+g, \dots ,d_k+g \}
\]
for all $g \in G$. The design developed by a $(v, k, \lambda)$-DS is actually a symmetric balanced incomplete block design with the same parameters~\cite[p.64]{Wallis}. \\

A difference set may not exist for an arbitrary choice of the parameters $(v, k, \lambda)$. But the following theorem shows the existence of a special type of a difference set, called a \textsl{Singer difference set}. As a result if $n$ is any prime power, then a finite projective plane of order $n$ exists.

\begin{theorem}\cite[Section 12.2]{Colburn}
If $n$ is a prime power, there is a cyclic difference set with parameters:
\begin{equation*}
\left( \frac{n^{d+1}-1}{n-1}, \frac{n^{d}-1}{n-1}, \frac{n^{d-1}-1}{n-1} \right).
\end{equation*}
These difference sets are called \textsl{Singer difference sets}.
\end{theorem}

To construct a Singer difference set with the parameters in the above theorem, let $\alpha$ be a generator of the multiplicative group of $\mathcal{F}_{n^{d+1}}$ ($\alpha$ is called primitive). Consider the usual trace function $Tr$ from  $\mathcal{F}_{n^{d+1}}$ to  $\mathcal{F}_{n}$, this function maps an element $x$ to $\sum_{i=0}^{d}x^{n^{i}}$. Then the following set is the desired difference set
\begin{equation}\label{DS}
\mathcal{D} = \{i\,|\,Tr(\alpha^{i}) = 0 \}.
\end{equation}
For more details see~\cite[Section 5.5]{Wallis} and~\cite[Section 12.2]{Colburn}. In particular, if $d=2$ and $n=2^{a}$, for every positive integer $a$, there exists a $(2^{2a}+2^{a}+1, 2^{a}+1, 1)$-DS and so a finite projective plane of the same parameters. This projective plane has $2^{2a}+2^{a}+1$ lines and each line has $2^{a}+1$ points on it.

\begin{definition}\cite[Section 4.1]{Wallis}
Consider an $(n+1)$-subset $\mathcal{O}$ of the point set in a finite projective plane with parameters $(n^{2}+n+1, n+1, 1)$. If the intersection of $\mathcal{O}$ with any line in this plane has at most 2 points in common, $\mathcal{O}$ is called an oval.
\end{definition}

If $\mathcal{D}$ is a Singer difference set for the finite projective plane with parameters $(2^{2a}+2^{a}+1, 2^{a}+1, 1)$-DS, then $-\mathcal{D}$ is an oval, and the set of lines is given by
\begin{equation}\label{Lines}
\mathcal{L}=\{ \mathcal{D}+x|x\in \mathbb{Z}_{2^{2a}+2^{a}+1} \}.
\end{equation}
\begin{lemma}\label{LinesWithZero} 
Let $\mathcal{P}$ be the finite projective plane with parameters $(2^{2a}+2^{a}+1, 2^{a}+1, 1)$ developed from the Singer difference set $\mathcal{D}$. Then in $\mathcal{P}$:
\begin{enumerate}[label=(\alph*)]
\item the lines containing the zero element are exactly the ones formed by adding an element of the oval $\mathcal{O} = -\mathcal{D}$ to the set $\mathcal{D}$;
\item for every $s\in \mathbb{Z}_{2^{2a}+2^{a}+1} \backslash \mathcal{O}$, there exists exactly one line which contains both $s$ and zero; and
\item each line containing zero, also contains exactly one element from $\mathcal{O}$.
\end{enumerate}
\end{lemma}

\begin{proof}
The lines containing the zero element are 
\[
\mathcal{D}^{*} = \{D+(-d_{1}), D+(-d_{2}), \cdots, D+(-d_{2^{a}+1}) \}.
\]
Also, for every $s\in \mathbb{Z}_{2^{2a}+2^{a}+1} \backslash \mathcal{O}$, there exists exactly one pair $(d_{i}, d_{j})$ in $D$ such that $s=d_{i}-d_{j}$ (by the definition of the Singer difference set and the fact that here $\lambda = 1$). Now, the line $\mathcal{D}+(-d_{j})$ is
\begin{equation*}
\mathcal{D}+(-d_{j}) = \{d_{1}-d_{j}, \cdots, \underbrace{d_{i}-d_{j}}_{s}, \cdots, \underbrace{d_{j}-d_{j}}_{0}, \cdots, d_{2^{a}+1}-d_{j}  \}.
\end{equation*}
So $s\in \mathcal{D} +(-d_{j})$, and $\mathcal{D} +(-d_{j})$ is the only line in $\mathcal{D}^{*}$ containing $s$, since otherwise there would be two lines in $\mathcal{D}^{*}$ containing both $s$ and 0. On the other hand, for every $d_{t} \in \mathcal{D}$, there is exactly one pair $(d_{i}, d_{j})$ in $\mathcal{O}$ for which $-d_{t} = d_{i}-d_{j}$, thus $-d_{t} \in \mathcal{D} -d_{j}$. So each line in $\mathcal{D}^{*}$ includes exactly one element of $\mathcal{O}$.\\
\end{proof}

\subsection{Construction of large cliques}

We know that $X_2(2k)$ is a vertex-transitive graph. We  apply the following well-known bound on the size of a coclique on a vertex-transitive graph.

\begin{theorem}[Clique-Coclique Bound]\cite[p.26]{GMB} Let $X$ be a vertex-transitive graph. Then
\begin{equation*}
\alpha(X) \,  \omega(X)\leq |V(X)|.
\end{equation*}
\\
If equality holds for a clique $C$ and a coclique $S$, then the vectors 
\begin{equation*}
\chi_{C}-\frac{|C|}{v}\mathbf{1}, \quad \chi_{S}-\frac{|S|}{v}\mathbf{1}
\end{equation*}
are orthogonal, where $\chi_{C}$ and $\chi_{S}$ are the characteristic vectors of $C$ and $S$ respectively. In particular $\chi_{C}^{T}\chi_{S} = 1$.
\end{theorem}

\begin{theorem}\label{M} 
Suppose that $X$ is a graph in a symmetric association scheme. If there is a clique in $X$ for which equality holds in the clique-coclique bound, then there exists a weighted adjacency matrix of $X$, for which the ratio bound holds with equality.
\end{theorem}
\begin{proof}
Let $\mathcal{A} = \{A_0, A_1, \dots, A_k\}$ be a symmetric association scheme on $v$ vertices and denote the row sum of $A_i$ by $v_i$. Let $X_{i}$ be the graph associated with $A_{i}$, for $i \in \{0,\dots, k\}$. Then $X$ is a graph such that $X = \bigcup_{i\in T} X_{i}$ for some $T \subset \{1,\dots, k\}$. 

Let $C$ be a clique in $X$ for which the clique-coclique bound holds with equality, so $\alpha(X) = v /|C|$. Let $\chi_C$ be the characteristic vector of $C$. Then $\chi_C^T \chi_C$ is a positive semi-definite (psd) matrix. The projection of this matrix into the Bose-Mesner algebra is the matrix
\begin{equation*}
\widehat{M} = \sum_{i=0}^k \frac{\chi_C^T A_i \chi_C }{v v_i}A_i.
\end{equation*}
The projection of a psd matrix is a psd matrix, so $\widehat{M}$ is again a psd matrix (see~\cite[Lemma 3.2]{MR3990672}). Since $A_0=I$, we have $\frac{\chi_C^T A_0 \chi_C }{v v_0}A_0 = \frac{|C|}{v} I$.

Define $M = \widehat{M} - \frac{|C|}{v} I$. Since $\widehat{M}$ is psd, the minimal eigenvalue for $M$ is at least $ -\frac{|C|}{v}$. Further, since $C$ is a clique in $X$, any two vertices in $C$ will be related in some $A_i$ for $i \in T$, in particular $\chi_C A_j \chi_C =0$ for $j \not \in T$. Thus 
\begin{equation*}
M = \sum_{i=1}^k \frac{\chi_C^T A_i \chi_C }{v v_i}A_i =
\sum_{i \in T} \frac{\chi_C^T A_i \chi_C }{v v_i}A_i
\end{equation*}
and $M$ is a weighted adjacency matrix of $X$. 

The row sum of $\widehat{M}$ can be calculated as follows
\[
\widehat{M} {\bf 1} = \sum_{i=0}^k \frac{\chi_C^T A_i \chi_C }{v v_i}A_i  {\bf 1} 
                              = \frac{1}{v} \sum_{i=0}^k  \chi_C^T A_i \chi_C   {\bf 1} 
                              = \frac{1}{v}  \chi_C^T \left( \sum_{i=0}^k  A_i  \right) \chi_C   {\bf 1} 
                             = \frac{1}{v} \chi_C^T J \chi_C  {\bf 1}
                             = \frac{|C|^2}{v} {\bf 1}.
\]
This implies the row sum of $M$ is 
\begin{equation*}
d= \frac{|C|^2}{v} - \frac{|C|}{v}.
\end{equation*}
We also know the minimum eigenvalue $\tau$ is negative and at least $ -\frac{|C|}{v}$ and the maximum coclique has size $\frac{v}{|C|}$. The ratio bound, applied to $M$ gives
\begin{equation*}
\frac{v}{|C|} = \alpha(X) \leq \frac{v}{1-\frac{d}{\tau}} = \frac{v}{1-\frac{ \frac{|C|^2}{v} - \frac{|C|}{v}}{\tau}}. 
\end{equation*}
Rearranging verifies that $\tau$ is less than or equal to $-\frac{|C|}{v}$. Thus $\tau = -\frac{|C|}{v}$ and equality holds in the ratio bound.
\end{proof} 


In the next theorem we  show that equality in the clique-coclique bound holds for $M_2(2k)$ for infinitely many values of $k$.

\begin{theorem}\label{CliqueApproach}
If $2k=2^{a}+2$, where $a$ is a positive integer, then $\omega(M_{2}(2k)) = (2k-1)(2k-3)$.
\end{theorem}

\begin{proof}
We  construct a maximum clique of size $(2k-1)(2k-3)$ for the graph $M_{2}(2k)$ where $2k=2^{a}+2$, for every positive integer $a$. This construction is described in the following four steps.

\begin{itemize}
\item[ {\bf Step 1.} ]Using  (\ref{DS}) and (\ref{Lines}) we  construct a Singer difference set $\mathcal{D}= \{d_{1}, d_{2}, \cdots, d_{2^{a}+1} \}$ with parameters $(2^{2a}+2^{a}+1, 2^{a}+1,1)$, and we  construct the lines $\mathcal{L}$ of the corresponding finite projective plane. Let $\mathcal{O} = -D= \{-d_{1}, -d_{2}, \cdots, -d_{2^{a}+1} \}$ be an oval of size $2^{a}+1$ in the projective plane. Note that by the definition of an oval, $\mathcal{O}$ has at most 2 points in common with each line in $\mathcal{L}$.\\

\item[{\bf Step 2.}]Let $K_{2k}$ be the complete graph with its vertices indexed by the elements in $\mathcal{O} \cup \{0\}$. Pick the lines from $\mathcal{L}$ which have at least one element in common with the oval $\mathcal{O}$, call this set $\mathcal{L'}$. For lines in $\mathcal{L}'$ with 2 points in common with the oval $\mathcal{O}$, consider their common pair of points as an edge in $K_{2k}$. Next consider the lines in $\mathcal{L}'$ with only one point in common with $\mathcal{O}$, these exist from Lemma~\ref{LinesWithZero}. The point in common with the oval $\mathcal{O}$ and the zero element on that line form an edge in $K_{2k}$.  Lemma~\ref{LinesWithZero} also indicates that for every $s\in \mathbb{Z}_{2^{2a}+2^{a}+1} \backslash \mathcal{O}$, there exists exactly one line in $\mathcal{L}$ containing both $s$ and zero. In addition, each line that contains 0 also contains exactly one element of $\mathcal{O}$.\\

\item[{\bf Step 3.}]Now consider $s\in \mathbb{Z}_{2^{2a}+2^{a}+1} \backslash \mathcal{O}$. We select the lines in $\mathcal{L'}$ which include $s$. The edges corresponding to the selected lines for $s$ form a perfect matching: by the definition of a finite projective plane every pair of lines has exactly one point in common, and since we are choosing the lines in the subset $\mathcal{L'}$ which contain $s$, they cannot have any other element in common. Hence the corresponding edges we are defining for $s$ are disjoint. Also note that by the argument in Step 2, there is exactly one edge with zero as one of its vertices, and all other selected lines include two elements of $\mathcal{O}$ (the lines in $\mathcal{L}'$ with one element of $\mathcal{O}$ are exactly the lines in $\mathcal{L}$ that include 0). In addition, for every $-d_{i}\in \mathcal{O}$, there is exactly one line that contains both $s$ and $-d_{i}$, which means the edges corresponding to the selected lines for $s$ cover all the elements of the oval $\mathcal{O}$ precisely once.\\

\item[{\bf Step 4.}] By repeating Step 3 for all $s \in \mathbb{Z}_{2^{2a}+2^{a}+1} \backslash \mathcal{O}$, we produce a set of perfect matchings which we denote by $\mathcal{C}$. The size of $\mathcal{C}$ is 
\begin{align*}
|\mathbb{Z}_{2^{2a}+2^{a}+1} \backslash \mathcal{O}| &= \left( 2^{2a}+2^{a}+1 \right) -(2^{a}+2) \\
&= (2^{a}+1)(2^{a}-1) \\
&= (2k-1)(2k-3).
\end{align*}
\end{itemize} 

The set $\mathcal{C}$ is actually a clique in $M_{2}(2k)$: in any clique in $M_{2}(2k)$, a pair of perfect matchings should have at most 1 edge in common. Suppose $PM_{s}$ and $PM_{t}$ are the two perfect matchings in $\mathcal{C}$ corresponding to $s$ and $t$ in $\mathbb{Z}_{2^{2a}+2^{a}+1}-\mathcal{O}$ and assume that they have 2 edges in common, say $e_{1}$ and $e_{2}$. This means that there are lines $\ell_{s_1}$ and  $\ell_{t_1}$ containing $s$ and $t$, respectively, and both $\ell_{s_1}$ and  $\ell_{t_1}$ include the vertices of $e_{1}$. By definition, every pair of lines has exactly one point in common so $\ell_{1}:=\ell_{s_1}=\ell_{t_1}$. Similarly, there is a line $\ell_{2}$ including the points on  $e_{2}$, $s$ and $t$. So the lines $\ell_{1}$ and $\ell_{2}$ have more than one point in common, therefore they are the same, which we label as $\ell$. But then the line $\ell$ contains more than two elements from the oval $\mathcal{O}$, which contradicts the definition of an oval.\\

Finally we note that $\mathcal{C}$ is a maximum clique in $M_{2k}$ since a canonical $2$-intersecting perfect matching is a coclique of size $(2k-5)!!$; hence the equality in the clique-coclique bound holds. 
\end{proof}

We have proved that for infinitely many values of $2k = 2^{a}+2$ we can build a maximum clique in $M_{2}(2k)$, hence we  can construct the matrix $M$ as Theorem~\ref{M} with equality in the ratio bound. The value $2^{a}+2$ grows fast as $a$ increases, which makes it difficult to find the maximum clique, and the values in $M$ by computer. In the next approach we  try to build the matrix $M$ without identifying maximum cliques.

\section{Second Approach }
\label{sec:Second Approach}

In this section we find a set of coefficients ${\bf a_{\lambda_{i}}}$ for the matrices $A_{\lambda_{i}}$ in the association scheme of perfect matchings so that the matrix $M = \sum_{i=1}^{m}{\bf a_{\lambda_{i}}} A_{\lambda_{i}}$ is a weighted adjacency matrix for the graph $M_{2}(2k)$ and the ratio bound holds with equality for the matrix $M$. To verify this, we need to determine the row sum and the least eigenvalue of the matrix $M$.\\

If we have the complete character table of the perfect matching association scheme, then we can easily find the row sum and the least eigenvalue of $M$. The matrices in an association scheme are simultaneously diagonalizable, therefore they have common eigenspaces. This means that the eigenvalue of a linear combination of the matrices $A_{\lambda_{i}}$ corresponding to an eigenspace is actually the same linear combination of the eigenvalues of matrices $A_{\lambda_{i}}$ corresponding to the eigenspace. But finding the complete character table of this scheme for $2k\geq 40$ is still an unsolved problem. In his 1994 paper, Muzychuk~\cite{Mu} studied the eigenvalues of the association scheme of the symmetric group $\Sym(2k)$. The calculations are quite complicated and Muzychuk only found the eigenvalues up to $2k = 10$. More recently in 2018, Srinivasan~\cite{Sri1, Sri2} presented a recursive algorithm to find the character tables up to $2k=40$. 

In this work, we calculate several entries in the character table of the perfect matching association scheme for all values of $2k$. From these eigenvalues we will find an appropriate weighted adjacency matrix. To do this, we calculate the eigenvalues of some carefully chosen \textsl{quotient} graphs.

\subsection{Eigenvalues of perfect matching association scheme}\label{Mod of Char Table}

\begin{definition}\label{quotient}\cite[Section 2.2]{GMB}
Let $\pi = [\pi_{1},\pi_{2}, \ldots, \pi_{i}]$ be a set partition of the vertices of the graph $X$. This partition is \textsl{equitable} if the number of vertices in $\pi_{k}$ that are adjacent to a vertex in $\pi_{\ell}$ is determined only by $k$ and $\ell$, where $k,\ell\in \{1,\ldots, i\}$. If $\pi$ is an equitable partition of $X$, the \textsl{quotient} graph $X/\pi$ is a directed multi-graph with the parts of $\pi$ as its vertices, and if a vertex in $\pi_{\ell}$ has exactly $\nu$ neighbours in $\pi_{k}$, then $X/\pi$ has $\nu$ arcs from $\pi_{\ell}$ to $\pi_{k}$. 
\end{definition}

Quotient graphs are usually represented as a matrix, the rows and columns are index by the parts of the equitable partition and the 
$(\pi_k, \pi_\ell)$-entry is the number of edges from a vertex in $\pi_k$ to $\pi_\ell$.

Consider an integer partition $\lambda = [\lambda_{1},\lambda_{2}, \cdots, \lambda_{i}]$. The orbit partition formed by the Young subgroup $\Sym(\lambda) = \Sym(\lambda_{1}) \times \Sym(\lambda_{2}) \times \cdots \times \Sym(\lambda_{i})$ acting on the set of all perfect matchings of $K_{2k}$ (vertices of $M_{2}(2k)$) is an equitable partition. Hence for every class $\mu$ in the perfect matching association scheme, the quotient graph of $X_{\mu}$, with respect to this orbit partition, is well-defined. We denote this quotient graph by $X_\mu/\pi(\lambda)$. The eigenvalues of $X_\mu/\pi(\lambda)$ are also eigenvalues of the matrix $A_{\mu}$~\cite[p.28]{GMB}. Using this fact, we will construct some quotient graphs for several classes in the perfect matching association scheme to build a portion of the character table in the next subsection.\\

Let $A_{\lambda}$ be one of the matrices in the perfect matching association scheme. Let $\Sym(\lambda)$ be a Young subgroup and $A_{\ell}/\pi(\lambda)$ represent the corresponding quotient graph. Any eigenvector $\nu'$ of the quotient graph can be \textsl{lifted} to form an eigenvector $\nu$ for $A_{\ell}$ (the $P$-entry of $\nu$ is equal to the entry of $\nu'$ corresponding to the part that contains $P$). The groups $\Sym(\lambda)$ and $\Sym(n)$ both act on the cosets of $\Sym(n)/\left( \Sym(2) \wr \Sym(k) \right)$, and thus also act on the vector $\nu$ by permuting the indices. Since the entries of $\nu$ are constant on the orbits of $\Sym(\lambda)$, the vector $\nu$ is unchanged by the action of $\Sym(\lambda)$. Define
\[
V = \spanof\{ \sigma \nu : \sigma \in \Sym(n) \}.
\]
In particular, $V$ is the $\Sym(n)$-module generated by the action of $\Sym(n)$ on $\nu$. 

For two integer partitions $\mu =[\mu_1,\dots, \mu_\ell]$ and $\lambda=[\lambda_1,\dots, \lambda_k]$ of $n$, we say that $\mu \geq \lambda$ in the \textsl{dominance ordering} if $\mu_j = \lambda_j$, for all $j<i$ and $\mu_i > \lambda_i$ for some 
$i \in\{1,\dots, \min\{k,\ell\}\}$. We will use $\phi_\lambda$ to denote the character of $\Sym(n)$ associated to the partition $\lambda$.

The decomposition of the representation of $\Sym(n)$ induced by the trivial representation on a Young subgroup is well-known.

\begin{theorem}\label{Ind.Decomp.1}\cite[Chapter 12]{GMB}
If $\lambda \vdash n$, then
\[
ind_{\Sym(n)}(1_{\Sym(\lambda)}) = \phi_{\lambda}+\sum_{\mu>\lambda}^{}K_{\mu \lambda}\phi_{\mu},
\]
where $K_{\mu \lambda}$ is the \textsl{Kostka number} (see \cite[Section 12.5]{GMB} for more information).
\end{theorem}

The decomposition of the representation of $\Sym(n)$ induced by the trivial representation on $Sym(2)\wr \Sym(k)$ is also well-known.

\begin{lemma}\label{Ind.Decomp.2}\cite[Chapter 12]{GMB}
For integers $n$ and $k$ with $n\geq 2k$,
\[
ind_{\Sym(n)}(1_{\Sym(2)\wr \Sym(k)}) = \sum_{\lambda \vdash k} \phi_{2\lambda}.
\]
\end{lemma}
Our plan is to use the Young subgroup to form the quotient graphs. The eigenvalues of the quotient graphs  belong to the modules that are in both decompositions in Theorem~\ref{Ind.Decomp.1} and Lemma~\ref{Ind.Decomp.2}.

\begin{theorem}\label{module}
Assume that $\Sym(n)$ acts on the set $\Omega$, and that $A$ is the adjacency matrix for an orbital of the action of $\Sym(n)$ on $\Omega$. 
Let $\lambda \vdash n$ and $\pi$ be the orbit partition from the action of $\Sym(\lambda)$ on $\Omega$. If $\eta$ is an eigenvalue of the quotient graph $A/\pi$, then $\eta$ is an eigenvalue of $A$. Moreover, $\eta$ belongs to some $\Sym(n)$-module represented by the partition $\mu$ where $\mu \geq \lambda$ in the dominance ordering.
\end{theorem}
\begin{proof}
Let $\pi$ be the orbit partition of $\Sym(\lambda)$ acting on $\Omega$. Assume that $\nu'$ is an eigenvector of the quotient graph $A/\pi$ with eigenvalue $\eta$. The vector $\nu'$ can be lifted to an $\eta$-eigenvector of $A$, which we  denote by $\nu$.\\

The group $\Sym(n)$ and $\Sym(\lambda)$ both act on $\Omega$ and thus also act on the vector $\nu$ by permuting the entries. Since the entries of $\nu$ are constant on the orbits of $\Sym(\lambda)$, the vector $\nu$ is unchanged by the action of $\Sym(\lambda)$. Define a vector space $W$ to be the span of the vector $\nu$, so $\Sym(\lambda)$ fixes every element in $W$.\\

Set $V = \oplus_{\sigma \in \Sym(n)}  \sigma W$. Then $V$ is isomorphic to the module for the induced representation $\ind_{\Sym(n)}(1_{\Sym(\lambda)}) = \phi_{\lambda}+\sum_{\mu>\lambda}K_{\mu \lambda}\phi_{\mu}$. Clearly the vector $\nu \in V$ and since $\nu$ is an $\eta$-eigenvector there is a $\mu \geq \lambda$ so that the $\mu$-module is a subspace of the $\eta$-eigenspace.
\end{proof}

\begin{example}\label{ex:ep}
Consider the matrix $A_{[2k-4,4]}$ in the perfect matching association scheme. Note that in the graph corresponding to this matrix, two perfect matchings are adjacent if their union forms a $4$-cycle and a $(2k-4)$-cycle. The following matrix is the quotient graph corresponding to the group $\Sym(2k-2)\times \Sym(2)$. Denote this quotient graph by $A(X_{[2k-4,4]})/\pi([2k-2,2])$ (in this notation the first integer partition is the class in the perfect matching association scheme, the second partition is the Young subgroup used to form the partition of the vertices in the graph).
\[
\renewcommand*{\arraystretch}{1.5}
A(X_{[2k-4,4]})/[2k-2,2] =\left[
\begin{array}{c|c}
\bf{0} & k(k-1)(2k-6)!! \\ \hline
\frac{1}{2}k(2k-6)!! & \frac{1}{2}k(2k-3)(2k-6)!!
\end{array}\right].
\]

\vspace{0.5cm}
For the matrix $A(X_{[2k-4,4]})/\pi([2k-2,2])$ the all-ones vector $\bf{1}$ is an eigenvector corresponding to the largest eigenvalue, 
$k(k-1)(2k-6)!!$. This eigenvalue is actually the degree of $A_{[2k-4,4]}$, and by Theorem~\ref{module}, this eigenvalue corresponds to the $[2k]$-module in the character table. It is well-known that the trace of a matrix is equal to the sum of its eigenvalues, so by subtracting the degree from the trace, we find the second eigenvalue of this matrix which is $-\frac{1}{2}k(2k-6)!!$. Using Theorem~\ref{module}, and noting that the degree eigenvalue belongs to the $[2k]$-module, it is easy to deduce that the second eigenvalue belongs to the $[2k-2,2]$-module.\\

The quotient graph of $A_{[2k-4,4]}$ with the orbit partition formed by the action of the group $\Sym(2k-4)\times \Sym(4)$ on the perfect matchings is the following
\[
\renewcommand*{\arraystretch}{1.5}
\left[
\begin{array}{c|c|c}
2(2k-6)!! & (2k-4)!! & (k-1)(k-2)(2k-6)!! \\ \hline
\frac{1}{2}(2k-6)!! & \frac{1}{2}(5k-2)(2k-6)!! & \frac{1}{2}(2k^{2}-7k+1)(2k-6)!!\\ \hline
\frac{3}{2}(k-1)(2k-8)!! & 3(2k^{2}-7k+1)(2k-8)!! & (2k^{3}-14k^{2}+\frac{51}{2}k-\frac{3}{2})(2k-8)!!\\
\end{array}\right].
\vspace*{0.4cm}
\]
Denote this matrix by $A(X_{[2k-4,4]})/\pi([2k-4,4])$. This matrix yields the eigenvalues of $A_{[2k-4,4]}$ belonging to the modules $[2k]$, $[2k-2,2]$, and $[2k-4,4]$. From the matrix $A(X_{[2k-4,4]})/\pi([2k-2,2])$ we already have two eigenvalues of the matrix $A(X_{[2k-4,4]})/ \pi([2k-4,4])$, those corresponding to the modules $[2k]$ and $[2k-2,2]$. Hence by subtracting these two eigenvalues from the trace, we find a third eigenvalue of $A_{[2k-4,4]}$, the one corresponding to the module $[2k-4,4]$, and it is equal to $\frac{1}{2}(7k-15)(2k-8)!!$.\\
\end{example}

By finding several quotient graphs for the classes $[2k]$, $[2k-2,2]$, $[2k-4,4]$, and $[2k-6,6]$ in the perfect matching association scheme, we  construct a portion of the character table for the association scheme on the perfect matchings for $K_{2k}$ for any $k\geq 6$. By using Theorem~\ref{module} and the dominance ordering recursively to define the quotient graphs of each class in this association scheme, we can determine the eigenvalues that belong to some of the modules in the character table. These results are recorded in Table~\ref{CharTable}.
\FloatBarrier
\begin{sidewaystable}
\begin{centering}
\begin{tabular}[h]{N| m{1.4cm} ? m{1.8cm} | m{3.4cm} | m{4.4cm} | c | m{3.9cm}| m{0.4cm} |}
\hline
\rule{0pt}{20pt}& ~ & \textcolor{Red}{$[2k]$} & \textcolor{Red}{$[2k-2,2]$} & \textcolor{Red}{$[2k-4,4]$} & $[2k-4,2,2]$ & $[2k-6,6]$ & $\cdots$ \\[4ex]
\Xhline{5\arrayrulewidth}
\rule{0pt}{20pt}& $\mathbf{\chi}_{[2k]}$ & $\frac{(2k)!!}{2k}$ & $\frac{(2k)!!}{2(2k-2)}$ & $\frac{(2k)!!}{4(2k-4)}$ & $\frac{(2k)!!}{8(2k-4)}$ & $\frac{(2k)!!}{6(2k-6)}$ & $\cdots$ \\[4ex]
\hline
 
\rule{0pt}{20pt}& $\chi_{[2k-2,2]}$ & $-(2k-4)!!$ & $\frac{(2k-4)!!}{2}$ & $\frac{-2k(2k-6)!!}{4}$ & \textcolor{Gray}{$?$} & $\frac{-2k(2k-4)!!}{6(2k-6)}$ & $\cdots$ \\[4ex]
  
\rule{0pt}{20pt}& $\chi_{[2k-4,4]}$ & $-(2k-6)!!$ & $-(5k-12)(2k-8)!!$ & $\frac{(7k-15)(2k-8)!!}{2}$ & \textcolor{Gray}{$?$} & $\frac{-2k(2k-6)!!}{6(2k-6)}$ & $\cdots$ \\[4ex]
 
\rule{0pt}{20pt}& $\chi_{[2k-4,2,2]}$ & $2(2k-6)!!$ & $-(2k-6)!!$ & $\frac{-(2k-6)!!}{2}$& \textcolor{Gray}{$?$} & $\frac{4k(2k-6)!!}{6(2k-6)}$ & $\cdots$ \\[4ex]

\rule{0pt}{20pt}& $\chi_{[2k-6,6]}$ & $-3(2k-8)!!$ & $-3(3k-10)(2k-10)!!$ & $-3(9k^{2}-71k+140)(2k-12)!!$ & \textcolor{Gray}{$?$} & $6(5k^{2}-38k+70)(2k-12)!!$ & $\cdots$ \\[4ex]

\rule{0pt}{20pt}&  \vdots & \vdots & \vdots & \vdots & \textcolor{Gray}{$\vdots$} & \vdots  & ~ \\[4ex]
\hline
\end{tabular}
\caption{Character table of the perfect matching association scheme}
\label{CharTable}
\end{centering}
\end{sidewaystable}
\FloatBarrier

We prove another result beyond Theorem~\ref{module} that will be used later in this work.
For a set $S \subset \{1, \dots ,2k \}$ of size four, consider the set of all perfect matchings that include two edges $e$ and $e'$ such that $e$ and $e'$ form a set partition of $S$.
This set is the first part in the equitable partition formed by the action of $\Sym([2k-4,4])$ on the perfect matchings given in Example~\ref{ex:ep}.
For example, if $S = \{1,2,3,4\}$ the characteristic vector of this set is 
\[
w_{\{1,2,3,4\}} = \nu_{\{1,2\}, \{3,4\} } +\nu_{\{1,3\}, \{2,4\} } + \nu_{\{1,4\}, \{2,3\} }.
\]
In the next theorem we use $w_S$ to denote the characteristic vector for the set of perfect matchings in which an arbitrary 4-set $S$ is contained in only 2 edges.

\begin{theorem}\label{thm:4sets}
The set $\{ w_S \, |\, S \subset \{1,2,\dots,2k\} \textrm{ with } |S|=4\}$ is a spanning set for the $\Sym(2k)$ module $\spanof \{ [2k], [2k-2,2], [2k-4,4]\}$.
\end{theorem}
\begin{proof}
Consider the quotient matrix $A(X_{[2k-4,4]})/\pi([2k-4,4])$ formed by the action of $\Sym([2k-4,4])$ on the perfect matchings (this matrix is given in Example~\ref{ex:ep}). The first part in the equitable partition is the set of all perfect matchings for which a fixed set of size $S$ four is contained in only two edges (that is, $S$ is the 4-set stabilized by $\Sym([2k-4,4])$).

The matrix $A(X_{[2k-4,4]})/\pi([2k-4,4])$ is a $3 \times 3$-matrix that is diagonalizable. This implies the vector $(1,0,0)$ can be expressed as a linear combination of the eigenvectors of the quotient matrix. Further, each of the eigenvectors of the quotient matrix can be lifted to be an eigenvector for the adjacency matrix. By Theorem~\ref{module}, these lifted vectors are the eigenvectors belonging to the $[2k]$, $[2k-2,2]$ and $[2k-4,4]$ modules of $\Sym(2k)$.

Using the same linear combination to produce the vector $(1,0,0)$ from the eigenvectors of the quotient matrix, $w_S$ is a linear combination of eigenvectors for the modules $[2k]$, $[2k-2,2]$ and $[2k-4,4]$ (indeed $w_S$ is the vector formed by lifting $(1,0,0)$).  So we  conclude that for any subset $S$ of size four that $w_S$ is in the span of the $[2k]$, $[2k-2,2]$ and $[2k-4,4]$ modules. 

Finally, we  show that the dimension of the span of $w_S$ where $S$ is taken over all 4-subsets of $\{1,2,\dots,2k\}$ is equal to the dimension of the span of the $[2k]$, $[2k-2,2]$ and $[2k-4,4]$ modules. Define $N$ to be the matrix with the rows indexed by the perfect matchings of $K_{2k}$ and the columns by the 4-subsets $S \subset \{1,2,\dots,2k\}$, with each column the vector $w_S$. The entries of  $N^TN$ depend only the size of the intersection of the 4-subsets, so it can be written as a linear combination of the matrices in the Johnson scheme $J(2k,4)$. In particular
\[
N^TN = (2k-5)!! I + (2k-7)!! J(2k,4,2) + 9*(2k-9)!! J(2k,4,0). 
\]
The eigenvalues of the Johnson scheme are well-known and can be used to calculate the eigenvalues of $N^TN$. It is straight-forward to see that 0 is an eigenvalue with multiplicity $2k-1 + \binom{2k}{4} - \binom{2k}{3}$. This implies the rank of $N^TN$, and hence the rank of $N$, equals the dimension of $[2k]$, $[2k-2,2]$ and $[2k-4,4]$ modules. Thus the set $\{ w_S \, | \, S \subset \{1,\dots ,2k\} \textrm{ with } |S| =4\}$ is a spanning set. 
\end{proof}

\subsection{The degrees of the irreducible modules of $\Sym(n)$}

In this subsection we review some results on the dimension the irreducible modules of $\Sym(2k)$. Later we  use these results to prove Theorem~\ref{LeastEvalTrick}.\\

Let $\lambda = [\lambda_{1}, \lambda_{2}, \cdots, \lambda_{t}]$ be an integer partition of $2k$, the dimension of the $\lambda$ module will be denoted by $m(\lambda)$. The \textsl{dual partition}
$\lambda^{*}$ to the partition $\lambda$ is the partition with the Young diagram that is the reflection of the Young diagram of $\lambda$. The degree of a partition and its dual is the same; denote by $m(\lambda) = m(\lambda^{*})$. A partition $\lambda$ is called \textsl{primary} if $\lambda \geq \lambda^{*}$ in the dominance ordering (see~\cite{Ra} for more details).

\begin{theorem}\cite[p.151]{Ra}\label{RaTheorem}
Let $\lambda = [ \lambda_{1},\lambda_{2},\cdots, \lambda_{t}]$ be an integer partition of $2k$ in which $\lambda_{1} \geq k$. Then,
\begin{equation*}
m([\lambda_{1}, 2k-\lambda_{1}])\leq m(\lambda).
\end{equation*}
\end{theorem}

The next result follows from a straight-forward application of the hook length formula.
\begin{lemma}\cite[Section 12.6]{GMB}\label{Hook}
Let $n \geq 2k$, then 
\begin{equation*}
m([n-k,k])\ = \binom{n}{k}-\binom{n}{k-1}.
\end{equation*}
\end{lemma}

The next result is a general bound on the degree of a representation in which the first part of the corresponding integer partition is considered small.

\begin{theorem}\cite[p.163]{Ra}\label{F(n)Theorem}
Let $\lambda$ be a primary partition of $n$ for which the first part $\lambda_{1}< \lfloor \frac{n}{2} \rfloor$. Then $m(\lambda)\geq F(n)$, where
\[
F(n) =
\begin{cases}
  n\cdot F(n-1)(m+2)   &  \textrm{ if $n=2m+1$ is odd},\\
  2\cdot F(n-1)      & \textrm{ if $n$ is even},
\end{cases}
\]
with $F(0) = 2$. In particular, for $n\geq 8$,
\begin{equation}\label{F(n)}
\frac{3}{2}\cdot F(n-1)\leq F(n) \leq 2\cdot F(n-1).
\end{equation}
\end{theorem}

\subsection{Least eigenvalue of weighted adjacency matrix}

In this subsection, our goal is to show that the set of perfect matchings with two fixed edges is a maximum coclique in $M_2(2k)$. To address this, we  determine an appropriate set of coefficients $a_{2\lambda}$ so that 
\begin{align}
M = \sum_{\lambda \vdash k} a_{2\lambda} A_{2\lambda}
\end{align}
is a weighted adjacency matrix of $M_2(2k)$ with row sum $(2k-1)(2k-3)-1$ and the least eigenvalue $-1$. This proves that the ratio bound 
\[
\alpha(X) = \frac{|V(X)|}{1-\frac{d}{\tau}} = \frac{(2k-1)!!}{1-\frac{(2k-1)(2k-3)-1}{-1}} = (2k-5)!!
\]
holds with equality for $M_2(2k)$. To be a weighted adjacency matrix of $M_2(2k)$, we need that $a_{2\lambda}  =0$ whenever $\lambda$ has 2 or more ones. Further, the eigenvalue of $M$ corresponding to the $\mu$-module is $\xi^\mu = \sum_{\lambda \vdash k} a_{2\lambda} \xi_{2\lambda}^\mu$, where $\xi_{2\lambda}^\mu$ is the eigenvalue of $A_{2\lambda}$ belonging to the $\mu$-module. 

\begin{theorem}\label{SmallValues}
For $3\leq k \leq 9$, there exists a weighted adjacency matrix of the graph $M_{2}(2k)$ for which the degree and the least eigenvalue are $(2k-1)(2k-3)-1$ and $-1$, respectively. 
\end{theorem}
\begin{proof}
For $k=3,4,5$, define the matrices $M_{6} = A_{[6]}+A_{[4,2]}$, $M_{8} = \frac{1}{4}A_{[8]}+\frac{1}{2}A_{[6,2]}+\frac{1}{2}A_{[4,4]}$, and $M_{10} = \frac{1}{12}A_{[10]}+\frac{1}{12}A_{[8,2]}+\frac{1}{6}A_{[4,4,2]}$. The matrices $M_{6}$, $M_{8}$, and $M_{10}$ are the desired weighted adjacency matrices for the graphs $M_{2}(6)$, $M_{2}(8)$, and $M_{2}(10)$, respectively. 

\FloatBarrier
\begin{center}
\begin{tabular}[h]{N | m{1.4cm} ? c | c | c | c | c ? c |}
\hline
\rule{0pt}{20pt}& ~ & \textcolor{Red}{$A_{[8]}$} & \textcolor{Red}{$A_{[6,2]}$} & \textcolor{Red}{$A_{[4,4]}$} & $A_{[4,2,2]}$ & $A_{[2,2,2,2]}$ & \textcolor{Green}{$M_{8}$} \\
\Xhline{5\arrayrulewidth}
\rule{0pt}{20pt}& $\mathbf{\chi}_{[8]}$ & $48$ & $32$ & $12$ & $12$ & $1$ & $\textcolor{Green}{34}$ \\
\hline
 
\rule{0pt}{20pt}& $\mathbf{\chi}_{[6,2]}$ & $-8$ & $4$ & $-2$ & $5$ & $1$ & $\textcolor{Green}{-1}$ \\
  
\rule{0pt}{20pt}& $\mathbf{\chi}_{[4,4]}$ & $-2$ & $-8$ & $7$ & $2$ & $1$ & $\textcolor{Green}{-1}$ \\
 
\rule{0pt}{20pt}& $\mathbf{\chi}_{[4,2,2]}$ & $4$ & $-2$ & $-2$ & $-1$ & $1$ & $\textcolor{Green}{-1}$ \\

\rule{0pt}{20pt}& $\mathbf{\chi}_{[2,2,2,2]}$ & $-6$ & $8$ & $3$ & $-6$ & $1$ & $\textcolor{Green}{4}$ \\
\hline
\end{tabular}
\vspace*{0.2cm}
\captionof{table}{Character table for $2k=8$}
\label{CharTableK=4}
\end{center}
\FloatBarrier
As we  see in Table~\ref{CharTableK=4}, the row sum and the least eigenvalue of $M_{8}$ are $(2k-1)(2k-3)-1$ and $-1$. Similarly, using the character tables for $k=3$ and $k=5$~\cite{Mu, Sri1}, we find the eigenvalues of the matrices $M_{6}$ and $M_{10}$ and verify that the ratio bound holds with equality.\\

For $k=6$, Theorem~\ref{CliqueApproach} proves that equality holds in the ratio bound. For $7\leq k \leq 9$, we have the complete character table for the perfect matching association scheme. So we can express the eigenvalues of $M = \sum_{\lambda \vdash k} a_{2\lambda} A_{2\lambda}$ as a system of linear equations. The objective is to maximize the value of the greatest eigenvalue (this is the row sum, so the eigenvalue belonging to the $[2k]$ module) while fixing the eigenvalues corresponding to the modules $[2k-2,2]$, $[2k-4,4]$, and $[2k-4,2,2]$ to be  $-1$, and having all other eigenvalues strictly greater than $-1$. The Gurobi Optimizer~\cite{gurobi} is then used to find solutions for these system of inequalities.
As such we determined the desired weighted adjacency matrices as follows:
\[
\begin{split}
M_{7} &= \frac{1}{640}A_{[14]}+\frac{1}{80}A_{[6,6,2]}+\frac{1}{60}A_{[4,4,4,2]}, \\
M_{8} &= \frac{1}{3840}A_{[14,2]}+\frac{1}{2048}A_{[10,6]}+\frac{1}{120}A_{[4,4,4,4]}, \\
M_{9} &= \frac{1}{80640}A_{[18]}+\frac{1}{13440}A_{[8,8,2]}+\frac{1}{4480}A_{[6,6,4,2]}.
\end{split}
\]
\end{proof}

To find the set of coefficients for $k\geq 10$, we consider linear combinations of the form
\[
M_{k} = \mathbf{a_1} A_{[2k]} + \mathbf{a_2} A_{[2k-2,2]} + \mathbf{a_3} A_{[2k-4,4]}.
\]
To find the values of $\mathbf{a_1}, \mathbf{a_2}$ and $\mathbf{a_3}$ for $k\geq 10$, we use the eigenvalues in the partial character table in Table~\ref{CharTable} to produce a corresponding linear
system. For this system, there is one equation for each of the eigenvalues that correspond to the irreducible modules $[2k-2,2]$, $[2k-4,4]$, and $[2k-4,2,2]$, which are equated to $-1$. The rationale for choosing these modules is that they, along with $[2k]$, are the modules that are in both the decomposition of $\ind_{\Sym(n)}(1_{\Sym([2k-4,2,2])})$ and $\ind_{\Sym(n)}(1_{\Sym(2) \wr \Sym(k)})$. Observe $\Sym([2k-4,2,2])$ is the group that stabilizes the set of all perfect matchings which include a fixed pair of edges.

Using the results in Subsection~\ref{Mod of Char Table}, this linear system becomes:

\newcolumntype{C}{>{{}}c<{{}}} 
\[
\setlength\arraycolsep{0pt}
\renewcommand\arraystretch{1.25}
\begin{array}{*{3}{rC}l}
   -(2k-4)!!\mathbf{a_1} & + &  (k-2)(2k-6)!!\mathbf{a_2} & - & k(k-3)(2k-8)!!\mathbf{a_3} & = & -1, \\
    -(2k-6)!!\mathbf{a_1} & - &  (5k-12)(2k-8)!!\mathbf{a_2} & + &  \frac{1}{2}(7k-15)(2k-8)!!\mathbf{a_3} & = &  -1, \\
    2(2k-6)!!\mathbf{a_1} & - & (2k-6)!!\mathbf{a_2} & - & \frac{1}{2}(2k-6)!!\mathbf{a_3} & = & -1.
\end{array}
\]

Solving this system, we obtain the coefficients $\mathbf{a_1} = \frac{1}{4(2k-6)!!}$, and $\mathbf{a_2} = \mathbf{a_3} = \frac{1}{(2k-6)!!}$. Note that for $k>4$, the determinant of the coefficient matrix corresponding to the aforementioned linear system is nonzero, so the values $\mathbf{a_1}$, $\mathbf{a_2}$, and $\mathbf{a_3}$ are unique. This can be easily checked  by using any basic mathematical software.
\begin{theorem}\label{LeastEvalTrick}
For $k\geq 10$, let $M = \mathbf{a_1}A_{[2k]}+\mathbf{a_2}A_{[2k-2,2]}+\mathbf{a_3}A_{[2k-4,4]}$
where $\mathbf{a_1} = \frac{1}{4(2k-6)!!}$, and $\mathbf{a_2} = \mathbf{a_3} = \frac{1}{(2k-6)!!}$. Then the row sum and the least eigenvalue of the matrix $M$ are $(2k-1)(2k-3)-1$ and $-1$, respectively. Moreover, the only modules with eigenvalue equal to -1 are $[2k-2,2]$, $[2k-4,4]$ and $[2k-4,2,2]$.
\end{theorem}

\begin{proof}
For $10\leq k\leq 14$, similar to the proof of Theorem~\ref{SmallValues}, by utilizing the complete character tables of the perfect matching association scheme~\cite{Mu, Sri1} we can find all the eigenvalues of the matrix $M$, and we see that the ratio bound holds with equality.\\

For the remainder of the proof assume that $k\geq 15$. Reviewing the linear system of equations, it follows that $-1$ is an eigenvalue of $M$ (corresponding to the modules $[2k-2]$, $[2k-4,4]$, and $[2k-4,2,2]$). Denote the row sum by $d_{M}$; this is the linear combination of the degrees for matrices $A_{[2k]}$, $A_{[2k-2,2]}$, and $A_{[2k-4,4]}$; say $d_{[2k]}$, $d_{[2k-2,2]}$, and $d_{[2k-4,4]}$. For the coefficients $\mathbf{a_1}$, $\mathbf{a_2}$, and $\mathbf{a_3}$, we  calculate
\begin{align*}
d_{M} &= \mathbf{a_1}d_{[2k]}+\mathbf{a_2}d_{[2k-2,2]}+\mathbf{a_3}d_{[2k-4,4]}\\ 
          &= \frac{(2k-2)!!}{4(2k-6)!!}+\frac{k(2k-4)!!}{(2k-6)!!}+\frac{k(k-1)(2k-6)!!}{(2k-6)!!} \\
          &= (2k-1)(2k-3)-1.\\
\end{align*}

Finally, we need to prove that all other eigenvalues of the matrix $M$ are strictly greater than $-1$. Let $\{d_{M}^{(1)}, -1^{(m_{1})}, -1^{(m_{2})}, -1^{(m_{3})},\theta_{4}^{(m_{4})}, \cdots,  \theta_{k}^{(m_{k})} \}$ be the spectrum of the matrix $M$, where the values $m_{i}$ represent the multiplicity of the eigenvalues. By the Lemma~\ref{Hook}, and the hook length formula~\cite[Section 12.6]{GMB} we have,
\begin{align*}
m_{1} =& \frac{2k(2k-3)}{2},\\
m_{2} =& \frac{2k(2k-1)(2k-2)(2k-7)}{4!} ,\\
m_{3} =& \frac{2k(2k-1)(2k-4)(2k-5)}{12}.
\end{align*}

As we defined, $M = \mathbf{a_1}A_{[2k]}+\mathbf{a_2}A_{[2k-2,2]}+\mathbf{a_3}A_{[2k-4,4]}$, the entry on the diagonal of the matrix $M^{2}$ is
\begin{align*}
M^{2}[i,i] &= \mathbf{a_1}^{2}d_{[2k]}+\mathbf{a_2}^{2}d_{[2k-2,2]}+\mathbf{a_3}^{2}d_{[2k-4,4]}\\
 &= \frac{(2k-2)!!}{\left( 4(2k-6)!!\right)^{2}}+ \frac{ k(2k-4)!!}{\left((2k-6)!!\right)^{2}}+\frac{ k(k-1)(2k-6)!!}{\left( (2k-6)!!\right)^{2}}\\
 &= \frac{13k^2-23k+2}{4(2k-6)!!}. 
\end{align*}
It is well-known that the trace of any matrix is equal to the sum of its eigenvalues. Hence we have that the trace of $M^2$ is
\begin{align}\label{TraceTrick}
\left( \frac{13k^2-23k+2}{4(2k-6)!!}\right) (2k-1)!! \nonumber = d_{M}^{2}+m_{1}+m_{2}+m_{3}+\sum_{i=4}^{k}m_{i}\theta_{i}^{2}.
\end{align}
From the equation above, we obtain the following inequality for any eigenvalue $\theta_i$ of $M$,
\begin{equation}\label{Lambda}
| \theta_{i} |\leq \sqrt{\frac{(2k-1)!!}{(2k-6)!!}\left( \frac{13k^2-23k+2}{4}\right)-\left( 18k^{4}-74k^{3}+\frac{191}{2}k^{2}-\frac{79}{2}k+4 \right)} \sqrt{\frac{1}{m_{i}}}.
\end{equation}

To finish the proof, it is sufficient to prove that the term in the right-hand side of the above inequality is strictly less than $1$; or equivalently 
\begin{equation}\label{mi}
m_{i} > \frac{(2k-1)!!}{(2k-6)!!}\left( \frac{13k^2-23k+2}{4}\right)-\left( 18k^{4}-74k^{3}+\frac{191}{2}k^{2}-\frac{79}{2}k+4 \right).
\end{equation}

Let the partition $\lambda_{i} = [\lambda_{i_1}, \lambda_{i_2},\cdots, \lambda_{i_\ell}]$ where $\lambda_{i_1}\geq \lambda_{i_2} \geq \cdots \geq \lambda_{i_\ell}$. Then there are 3 cases:
\begin{itemize}
\item[{\bf Case 1.}] Assume that $\lambda_{i_1}\geq k$. 

First, consider the module $[2k-6,6]$. Using the linear combination of the eigenvalues of the matrices $A_{[2k]}$, $A_{[2k-2,2]}$, and $A_{[2k-4,4]}$ corresponding to the module $[2k-6,6]$ it follows  that the eigenvalue of $M$ belonging to $[2k-6,6]$ is
\[
\frac{-6(2k-12)!!}{(2k-6)!!} \left( 8k^2-65k+130\right).
\]
This eigenvalue is greater than $-1$ for $k\geq 15$ (this can be checked with a mathematical software). 

Second, consider the modules $[2k-6,4,2]$ and $[2k-6,2,2,2]$. Using the hook length formula~\cite[Section 12.6]{GMB}, we calculate the dimensions $m_{[2k-6,4,2]}$ and $m_{[2k-6,2,2,2]}$.  If, in the right hand side of (\ref{mi}), we approximate the term $(2k-6)!!$ with $(2k-7)!!$, then the inequality holds for $m_{[2k-6,4,2]}$ with $k\geq 37$, and for $m_{[2k-6,2,2,2]}$ for $k\geq 63$. Using Maple for the values of $m_{[2k-6,4,2]}$ for $15\leq k \leq 36$, and for and for $m_{[2k-6,2,2,2]}$  with $15\leq k \leq 62$, observe that (\ref{mi}) holds for these two modules.\\

Next consider the module $[2k-8,8]$. By Lemma~\ref{Hook} the multiplicity of the corresponding eigenvalue is 
\begin{equation*}
\frac{(2k)(2k-1)(2k-2)(2k-3)(2k-4)(2k-5)(2k-6)(2k-15)}{8!}.
\end{equation*}
Thus, if in the right hand side of (\ref{mi}), we approximate the term $(2k-6)!!$ with $(2k-7)!!$, then the inequality holds for all $k\geq 20$. Using Maple for the values $15\leq k\leq 19$, we have that (\ref{mi}) holds for this module.\\

Finally, using Theorem~\ref{RaTheorem} and Lemma~\ref{Hook}, for all modules with $\lambda_{i_1} \geq k$ and $\lambda_{i_1}\leq 2k-8$, the multiplicity $m_{i}$ is greater than the multiplicity for the module $[2k-8,8]$. Thus (\ref{mi}) holds for all remaining modules with $\lambda_{i_1}>k$ .

\item[{\bf Case 2.}] Assume that $\lambda_{i_1} < k$ and $\lambda_{i}$ is primary. 

Using Theorem~\ref{F(n)Theorem} for $2k\geq 8$, we have
\begin{equation}\label{UpBound}
m(\lambda)\geq F(2k) = 2F(2k-1) \geq (2)     \left( \frac{3}{2} \right) F(2k-2)\geq \cdots \geq 2(3^{k}).
\end{equation}
Approximating the term $(2k-6)!!$ with $(2k-7)!!$, in (\ref{Lambda}), for $k\geq 15$ we have
\begin{equation*}
\theta_{i}^{2}  \leq \frac{104k^{5}-724k^{4}+1738k^{3}+595k-46}{4m_{i}} \leq \frac{104k^{5}}{8(3^{k})}\leq 1.
\end{equation*}
In fact, for $k\geq 1$, the term $(-724k^{4}+1738k^{3}+595k-46)$ is always negative. Noting this along with (\ref{UpBound}) proves the second inequality above. The last inequality holds for all $k\geq 15$.\\

\item[{\bf Case 3.}] Assume that $\lambda_{i_1} < k$ and $\lambda_{i}$ is not primary. 

We know that the degrees of a partition and its dual are the same, so $m(\lambda) = m(\lambda^{*})$. Let $\lambda^{*} = (\lambda_{1}^{*},\lambda_{2}^{*},\cdots, \lambda_{t}^{*})$. Since $\lambda$ is an even partition, then $\lambda_{1}^{*} = \lambda_{2}^{*} = f$. Since  $\lambda_{i}$ is not primary, $\lambda^{*}\geq \lambda$. This means that $\lambda^{*}$ is primary. If $f\geq k$, then $\lambda^{*} = [k,k]$ which is covered in Case 1. If $f<k$, this is covered by Case 2. 
\end{itemize}
\end{proof}


Recall that a canonical $2$-intersecting set of perfect matchings is the set of all perfect matchings that contain the edges $e_1$ and $e_2$. The size of a canonical $2$-intersecting set is $(2k-5)!!$. Theorems~\ref{SmallValues} and~\ref{LeastEvalTrick}, along with the ratio bound (Theorem~\ref{ratioBound}), show that the size of a 2-intersecting set of perfect matchings is no larger than $(2k-5)!!$. The ratio bound further implies that if $S$ is a maximum 
2-intersecting set and $v_S$ is the characteristic vector of $S$, then $v_S - \frac{1}{(2k-1)(2k-3)}{\bf 1}$ is a $-1$-eigenvector for $M_2(2k)$. Since the only irreducible representations of $\Sym(2k)$ that affords $-1$ as an eigenvalue are $[2k-2,2]$, $[2k-4,4]$, and $[2k-4,2,2]$, this implies that $v_S$ is in $\spanof \{[2k], [2k-2,2], [2k-4,4], [2k-4,2,2]\}$. The next result gives a more convenient spanning set for this space. 
Define $\nu_{e_{1},e_{2}}$ be the characteristic vector of all perfect matchings with the (disjoint) edges $e_{1}$ and $e_{2}$.

\begin{theorem}
For $k\geq 4$, let
\[
V(2k) = \spanof\{\nu_{e_{1},e_{2}} \,|\, e_{1},e_{2} \textrm{ edges in } K_{2k} \} 
\]
and 
\[
W(2k) = \spanof \{[2k], [2k-2,2], [2k-4,4], [2k-4,2,2]\}.
\]
Then $W(2k) = V(2k)$.
\end{theorem}

\begin{proof}
Using the ratio bound for the weighted adjacency matrices given in Theorems~\ref{SmallValues} and~\ref{LeastEvalTrick}, we  see that the size of a maximum coclique in $M_{2}(2k)$ is $(2k-5)!!$, so all perfect matchings which have two fixed edges, say $e_{1}$ and $e_{2}$, form a maximum coclique. Thus, by Theorem~\ref{ratioBound}, $\nu_{e_{1},e_{2}}-\frac{(2k-5)!!}{(2k-1)!!}\mathbf{1}$ is an eigenvector for the least eigenvalue. The least eigenvalue is $-1$ and only the modules $[2k-2,2]$, $[2k-4,4]$, and $[2k-4,2,2]$ have $-1$ as their eigenvalues. So $\nu_{e_{1},e_{2}}\in W(2k)$, and $V(2k) \subseteq W(2k)$. 

Denote the dimensions of the sets $W(2k)$ and $V(2k)$, by $D_{W}(2k)$ and $D_{V}(2k)$. Then, by Lemma~\ref{Hook}, and more generally the hook length formula
\[
D_{W}(2k) = 1+\binom{2k}{2}-\binom{2k}{1}+\binom{2k}{4}-\binom{2k}{3}+\frac{(2k)(2k-1)(2k-4)(2k-5)}{12}.
\]
For $4\leq k\leq 11$, using GAP~\cite{GAP4}, we note $D_{V}(2k) = D_{W}(2k)$, thus $V(2k)=W(2k)$. For $k>11$, we prove the same result by induction. 

By summing all vectors in $V(2k+2)$, we obtain a multiple of the all ones vectors, so the module $[2k+2]$ is contained in $V(2k+2)$. Similarly, by summing all vectors with a fixed edge and subtracting an appropriate multiple of the all ones vector, we see the space $V(2k)$ contains the characteristic vector of the set of all perfect matchings that contain a fixed edge. Thus $V(2k+2)$ includes the module $[2k,2]$~\cite[Lemma 8.2]{GMP}.  

For any set $S \subset \{1,2,\dots 2k+2\}$ of size four the characteristic vector of all the perfect matchings in which the four elements of $S$ appear as two independent edges is in $V(2k+2)$. By Theorem~\ref{thm:4sets} these vectors are a spanning set for $\spanof \{ [2k+2], [2k,2], [2k-2,4]\}$. Thus each of these modules is contained in $V(2k)$.

Since the edges of $K_{2k}$ are a subset of the edges form $K_{2k+2}$, the space $V(2k+2)$ contains a subspace isomorphic to $V(2k)$. By induction, the dimension of $V(2k)$ is
\[
D(2k) = 1+\binom{2k}{2}-\binom{2k}{1}+\binom{2k}{4}-\binom{2k}{3}+\frac{(2k)(2k-1)(2k-4)(2k-5)}{12}.
\]
This is a lower bound on the dimension of $V(2k+2)$. 

For $k\geq 11$, 
\[
D_V(2k) > D([2k+2]) + D([2k,2]) + D([2k-2,4])
\]
We conclude that $V(2k+2)$ must include some of the space $[2k-2,2,2]$. But since these are irreducible $G$-modules and $V(2k+2)$ is invariant over $G$, this implies all of $[2k-2,2,2]$ is also contained in $V(2k+2)$.
\end{proof}

Putting these results together, we have our main result.

\begin{theorem}
The size of the largest set of $2$-intersecting perfect matchings in $K_{2k}$ with $k\geq 3$ is $(2k-5)!!$.
Further, if $S$ is a set of $2$-intersecting perfect matchings the characteristic vector of $S$ is a linear combination of the characteristic vectors of the canonically 2-intersecting sets of perfect matchings.
\end{theorem}

\section{Further Work}

In this paper, we proved that the Erd\H{o}s-Ko-Rado theorem holds for $2$-intersecting families of perfect matchings of the complete graph $K_{2k}$. Our first open question is if 
these approaches can be generalized
to prove a version of the Erd\H{o}s-Ko-Rado theorem for the family of $t$-intersecting perfect matchings of the complete graph $K_{2k}$, where $t>2$? This has been done for $k$ sufficiently large relative to $t$ in~\cite{L}. In this work, it is quite remarkable that we were able to find a weighted adjacency matrix for $M_2(2k)$ for which the ratio bound holds with equality that only uses three of the classes of the association scheme. It is an open question if a comparably simple weighted adjacency matrix would exist for larger values of $t$.\\

Our next question is, if we can characterize the largest set of  $t$-intersecting perfect matchings for graphs other than the complete graph? In the special case that the graph is the complete bipartite graph $K_{n,n}$, each perfect matching corresponds to a permutation. The set of intersecting permutations has been well-studied and versions of the EKR theorem hold~\cite{CaKu, Ellis}. It would be interesting to consider other bipartite graphs, such as the hypercube, although just enumerating the perfect matchings maybe quite difficult.\\

Finally, while we were working on computing various entries of the character table of the perfect matching association scheme, we observed some interesting patterns for the values in the table. There are conjectures and some results about signs and values of the eigenvalues in the association scheme for the permutations, see~\cite{Ku2}. We suspect that there are similar results for the eigenvalues in the perfect matching scheme. For example, we make the following related interesting conjecture.

\begin{conj}
Consider the character table of the perfect matching association scheme for $2k$. The greatest eigenvalue in the row corresponding to the module $[2k-2\ell, 2\ell]$ is the one that corresponds to the same class of the scheme, $[2k-2\ell,2\ell]$. In addition, in the same row all the eigenvalues corresponding to the classes which are greater than $[2k-2\ell,2\ell]$ (in the dominance ordering) are negative.
\end{conj}

\section{Acknowledgements}
We wish to thank Brett Stevens who introduced us to the clique used in Theorem~\ref{CliqueApproach}.

\bibliographystyle{plain}

\newpage
\appendix
\section{The Adjacency Matrices of Some Quotient graphs in the Perfect Matching Association Scheme}
\subsection*{The adjacency matrices of the quotient graphs corresponding to the group $\Sym(2k-2)\times \Sym(2)$}
\begin{center}
\begin{tabular}{N|M{2.7cm}|M{2.7cm}|}
\hline 
\vspace*{0.3cm}
\rule{0pt}{10pt}& $\bf{0}$ & $(2k-2)!!$ \\ [1ex] 
\hline
\vspace*{0.3cm}
\rule{0pt}{10pt}& $(2k-4)!!$ & $(2k-3)(2k-4)!!$ \\[1ex] 
\hline
\end{tabular}
\vspace*{0.2cm}
\captionof{table}{The adjacency matrix of $X_{[2k]}/[2k-2,2]$}
\end{center}
\begin{center}
\begin{tabular}{N|M{2.7cm}|M{2.7cm}|}
\hline 
\vspace*{0.3cm}
\rule{0pt}{13pt}& $(2k-4)!!$ & $(k-1)(2k-4)!!$ \\ [1ex] 
\hline
\vspace*{0.3cm}
\rule{0pt}{10pt}& $\frac{1}{2}(2k-4)!!$ & $\frac{1}{2}(2k-1)(2k-4)!!$\\[1ex] 
\hline
\end{tabular}
\vspace*{0.2cm}
\captionof{table}{The adjacency matrix of $X_{[2k-2,2]}/[2k-2,2]$}
\end{center}
\begin{center}
\begin{tabular}{N|M{2.7cm}|M{2.7cm}|}
\hline 
\vspace*{0.4cm}
\rule{0pt}{10pt}& $\bf{0}$ & $\frac{(2k)!!}{6(2k-6)}$ \\ [1ex] 
\hline
\vspace*{0.4cm}
\rule{0pt}{10pt}& $\frac{(2k)(2k-4)!!}{6(2k-6)}$ & $\frac{(2k)(2k-3)(2k-4)!!}{6(2k-6)}$\\[1ex] 
\hline
\end{tabular}
\vspace*{0.2cm}
\captionof{table}{The adjacency matrix of $X_{[2k-6,6]}/[2k-2,2]$}
\end{center}
\subsection*{The adjacency matrices of the quotient graphs corresponding to the group $\Sym(2k-4)\times \Sym(4)$}
\centering
\begin{tabular}{N| M{3.5cm} | M{3.9cm} | M{3.9cm} |}
\hline 
\vspace*{0.5cm}
\rule{0pt}{10pt}& $\bf{0}$ & $4(2k-4)!!$ & $(2k-6)(2k-4)!!$ \\ [1ex] 
\hline
\vspace*{0.5cm}
\rule{0pt}{10pt}& $2(2k-6)!!$ & $2(5k-12)(2k-6)!!$ & $(2k-6)(2k-5)(2k-6)!!$\\[1ex] 
\hline 
\vspace*{0.5cm}
\rule{0pt}{10pt}& $3(2k-6)!!$& $6(2k-5)(2k-6)!!$ & $(2k-7)(2k-5)(2k-6)!!$ \\[1ex] 
\hline 
\end{tabular}
\vspace*{0.2cm}
\captionof{table}{The adjacency matrix of $X_{[2k]}/[2k-4,4]$}
\centering
\begin{tabular}{N| M{3.5cm} | M{3.9cm} | M{3.9cm} |}
\hline 
\vspace*{0.5cm}
\rule{0pt}{10pt}& $\bf{0}$ & $4(2k-4)!!$ & $(k-4)(2k-4)!!$ \\ [1ex] 
\hline
\vspace*{0.5cm}
\rule{0pt}{10pt}& $2(2k-6)!!$ & $(7k-18)(2k-6)!!$ & $(2k^{2}-11k+16)(2k-6)!!$\\[1ex] 
\hline 
\vspace*{0.5cm}
\rule{0pt}{10pt}& $3(k-4)(2k-8)!!$& $6(2k^{2}-11k+16)(2k-8)!!$ & $(2k^{2}-9k+12)(2k-7)(2k-8)!!$ \\[1ex] 
\hline 
\end{tabular}
\vspace*{0.2cm}
\captionof{table}{The adjacency matrix of $X_{[2k-2,2]}/[2k-4,4]$}
\subsection*{The adjacency matrices of the quotient graphs corresponding to the group $\Sym(2k-6)\times \Sym(6)$}
\begin{centering}
\begin{tabular}{N| M{2cm} | M{3cm} | M{3.7cm} | M{3.7cm} |}
\hline 
\vspace*{0.6cm}
\rule{0pt}{15pt}& $\bf{0}$ & $24(2k-6)!!$ & $12(2k-8)(2k-6)!!$ & $(2k-8)(2k-10)(2k-6)!!$ \\ [1ex] 
\hline
\vspace*{0.6cm}
\rule{0pt}{15pt}& $8(2k-8)!!$ & $8(8k-27)(2k-8)!!$ & $2(13k-45)(2k-8)(2k-8)!!$ & $(2k-7)(2k-8)(2k-10)(2k-8)!!$\\[1ex] 
\hline 
\vspace*{0.6cm}
\rule{0pt}{15pt}& $12(2k-8)!!$ & $6(13k-45)(2k-8)!!$ & $4(7k-30)(2k-7)(2k-8)!!$ & $(2k-7)(2k-9)(2k-10)(2k-8)!!$\\[1ex] 
\hline 
\vspace*{0.6cm}
\rule{0pt}{15pt}& $15(2k-8)!!$ & $45(2k-7)(2k-8)!!$ & $15(2k-7)(2k-9)(2k-8)!!$ & $(2k-7)(2k-9)(2k-11)(2k-8)!!$\\[1ex] 
\hline 
\end{tabular}
\vspace*{0.2cm}
\captionof{table}{The adjacency matrix of $X_{[2k]}/[2k-6,6]$}
\end{centering}
\begin{centering}
\begin{tabular}{N| M{2cm} | M{3cm} | M{3.7cm} | M{3.7cm} |}
\hline 
\vspace*{0.6cm}
\rule{0pt}{15pt}& $\bf{0}$ & $24(2k-6)!!$ & $6(3k-14)(2k-6)!!$ & $(k-5)(2k-12)(2k-6)!!$ \\ [1ex] 
\hline
\vspace*{0.6cm}
\rule{0pt}{15pt}& $8(2k-8)!!$ & $4(13k-48)(2k-8)!!$ & $(34k^{2}-274k+564)(2k-8)!!$ & $2(2k^{3}-27k^{2}+123k-190)(2k-8)!!$\\[1ex] 
\hline 
\vspace*{0.6cm}
\rule{0pt}{15pt}& $6(3k-14)(2k-10)!!$ & $6(17k^{2}-137k+282)(2k-10)!!$ & $2(32k^{3}-390k^{2}+1627k-2334)(2k-10)!!$ & $(8k^{4}-136k^{3}+886k^{2}-2642k+3060)(2k-10)!!$\\[1ex] 
\hline 
\vspace*{0.6cm}
\rule{0pt}{15pt}& $15(k-6)(2k-10)!!$ & $45(2k^{2}-17k+38)(2k-10)!!$ & $15(4k^{3}-48k^{2}+203k-306)(2k-10)!!$ & $(8^{4}-132k^{3}+838k^{2}-2487k+2970)(2k-10)!!$\\[1ex] 
\hline 
\end{tabular}
\vspace*{0.2cm}
\captionof{table}{The adjacency matrix of $X_{[2k-2]}/[2k-6,6]$}
\end{centering}
\begin{centering}
\begin{tabular}{N| M{2cm} | M{3cm} | M{3.7cm} | M{3.7cm} |}
\hline 
\vspace*{0.8cm}
\rule{0pt}{15pt}& $\bf{0}$ & $12(2k-6)!!$ & $3(2k-8)(2k-6)!!$ & $(k^{2}-7k+12)(2k-6)!!$ \\ [1ex] 
\hline
\vspace*{0.8cm}
\rule{0pt}{15pt}& $4(2k-8)!!$ & $10(2k-6)!!$ & $(13k^{2}-79k+108)(2k-8)!!$ & $(2k^{3}-21k^{2}+65k-52)(2k-8)!!$\\[1ex] 
\hline 
\vspace*{0.8cm}
\rule{0pt}{15pt}& $3(2k-8)!!$& $3(13k^{2}-79k+108)(2k-10)!!$ & $(28k^{3}-274k^{2}+792k-576)(2k-10)!!$ & $(4k^{4}-60k^{3}+311k^{2}-609k+276)(2k-10)!!$ \\[1ex] 
\hline 
\vspace*{0.8cm}
\rule{0pt}{20pt}& $15(k^{2}-7k+12)(2k-12)!!$ & $45(2k^{3}-21k^{2}+65k-52)(2k-12)!!$ & $15(4k^{4}-60k^{3}+311k^{2}-609k+276)(2k-12)!!$ & $(8k^{5}-164k^{4}+1282k^{3}-4591k^{2}+6795k-1980)(2k-12)!!$ \\[1ex]
\hline
\end{tabular}
\vspace*{0.2cm}
\caption{The adjacency matrix of $X_{[2k-4,4]}/[2k-6,6]$}
\end{centering}
\begin{centering}
\begin{tabular}{N| M{2cm} | M{3cm} | M{3.7cm} | M{3.7cm} |}
\hline 
\vspace*{0.7cm}
\rule{0pt}{15pt}& $8(2k-8)!!$ & $4(2k-6)!!$ & $2(2k-2)(2k-6)!!$ & $\frac{2}{3}(k^{2}-6k+2)(2k-6)!!$ \\ [1ex] 
\hline
\vspace*{0.7cm}
\rule{0pt}{15pt}& $\frac{4}{3}(2k-8)!!$ & $(\frac{32}{3}k-4)(2k-8)!!$ & $(\frac{26}{3}k^{2}-\frac{116}{3}k+4)(2k-8)!!$ & $(\frac{4}{3}k^{3}-\frac{38}{3}k^{2}+\frac{92}{3}k-\frac{4}{3})(2k-8)!!$\\[1ex] 
\hline 
\vspace*{0.7cm}
\rule{0pt}{15pt}& $2(2k-2)(2k-10)!!$ & $(26k^{2}-116k+12)(2k-10)!!$ & $(\frac{56}{3}k^{3}-164k^{2}+\frac{1108}{3}k-12)(2k-10)!!$ & $(\frac{8}{3}k^{4}-\frac{112}{3}k^{3}+\frac{526}{3}k^{2}-\frac{836}{3}k+4)(2k-10)!!$\\[1ex] 
\hline 
\vspace*{0.7cm}
\rule{0pt}{23pt}& $10(k^{2}-6k+2)(2k-12)!!$ & $60(k^{3}-\frac{19}{2}k^{2}+23k-1)(2k-12)!!$ & $40(k^{4}-14k^{3}+\frac{263}{4}k^{2}-\frac{209}{2}k+\frac{3}{2})(2k-12)!!$ & $(\frac{16}{3}k^{5}-104k^{4}+\frac{2284}{3}k^{3}-2486k^{2}+\frac{9220}{3}k-20)(2k-12)!!$\\[1ex] 
\hline 
\end{tabular}
\vspace*{0.2cm}
\captionof{table}{The adjacency matrix of $X_{[2k-6,6]}/[2k-6,6]$}
\end{centering}
\subsection*{The adjacency matrices of the quotient graphs corresponding to the group $\Sym(2k-6)\times \Sym(6)$}
\begin{centering}
\begin{tabular}{N| M{1.5cm} | M{2cm} | M{2cm} | M{2cm}| M{2cm}| M{2.4cm} |}
\hline 
\vspace*{0.7cm}
\rule{0pt}{13pt}&$\bf{0}$ & $\bf{0}$ & $\bf{0}$ & $\bf{0}$ & $4(2k-4)!!$ & $(2k-6)(2k-4)!!$ \\ [1ex] 
\hline
\vspace*{0.7cm}
\rule{0pt}{13pt}&$\bf{0}$ & $\bf{0}$ & $(2k-4)!!$ & $2(2k-6)!!$ & $4(2k-5)(2k-6)!!$ & $(2k-5)(2k-6)(2k-6)!!$ \\ [1ex] 
\hline 
\vspace*{0.7cm}
\rule{0pt}{13pt}&$\bf{0}$ & $(2k-4)!!$ & $\bf{0}$ & $2(2k-6)!!$ & $4(2k-5)(2k-6)!!$ & $(2k-5)(2k-6)(2k-6)!!$ \\ [1ex] 
\hline 
\vspace*{0.7cm}
\rule{0pt}{13pt}&$\bf{0}$ & $(2k-4)!!$ & $(2k-4)!!$ & $\bf{0}$ & $2(2k-4)!!$ & $(2k-6)(2k-4)!!$ \\ [1ex] 
\hline 
\vspace*{0.7cm}
\rule{0pt}{13pt}&$(2k-6)!!$ & $(2k-5)(2k-6)!!$ & $(2k-5)(2k-6)!!$ & $(2k-6)!!$ & $(6k-14)(2k-6)!!$ & $(2k-5)(2k-6)(2k-6)!!$ \\ [1ex] 
\hline 
\vspace*{0.7cm}
\rule{0pt}{13pt}&$(2k-6)!!$ & $(2k-5)(2k-6)!!$ & $(2k-5)(2k-6)!!$ & $2(2k-6)!!$ & $4(2k-5)(2k-6)!!$ & $(2k-5)(2k-7)(2k-6)!!$ \\ [1ex] 
\hline 
\end{tabular}
\vspace*{0.5cm}
\captionof{table}{The adjacency matrix of $X_{[2k]}/[2k-4,2,2]$}
\end{centering}
\begin{sidewaystable}
\begin{centering}
\begin{tabular}{N| M{1.7cm} | M{2.7cm} | M{2.7cm} | M{2.7cm}| M{2.7cm}| M{3.5cm} |}
\hline 
\vspace*{0.7cm}
\rule{0pt}{13pt}&$\bf{0}$ & $(2k-4)!!$ & $(2k-4)!!$ & $\bf{0}$ & $2(2k-4)!!$ & $(k-4)(2k-4)!!$ \\ [1ex] 
\hline
\vspace*{0.7cm}
\rule{0pt}{13pt}&$(2k-6)!!$ & $(2k-5)(2k-6)!!$ & $(k-3)(2k-6)!!$ & $(2k-6)!!$ & $2(2k-5)(2k-6)!!$ & $(2k^{2}-11k+16)(2k-6)!!$ \\ [1ex] 
\hline 
\vspace*{0.7cm}
\rule{0pt}{13pt}&$(2k-6)!!$ & $(k-3)(2k-6)!!$ & $(2k-5)(2k-6)!!$ & $(2k-6)!!$ & $2(2k-5)(2k-6)!!$ & $(2k^{2}-11k+16)(2k-6)!!$ \\ [1ex] 
\hline 
\vspace*{0.7cm}
\rule{0pt}{13pt}&$\bf{0}$ & $\frac{1}{2}(2k-4)!!$ & $\frac{1}{2}(2k-4)!!$ & $\bf{0}$ & $3(2k-4)!!$ & $(k-4)(2k-4)!!$ \\ [1ex] 
\hline 
\vspace*{0.7cm}
\rule{0pt}{13pt}&$\frac{1}{2}(2k-6)!!$ & $\frac{1}{2}(2k-5)(2k-6)!!$ & $\frac{1}{2}(2k-5)(2k-6)!!$ & $\frac{3}{2}(2k-6)!!$ & $(5k-13)(2k-6)!!$ & $(2k^{2}-11k+16)(2k-6)!!$ \\ [1ex] 
\hline 
\vspace*{0.7cm}
\rule{0pt}{13pt}&$(k-4)(2k-8)!!$ & $(2k^{2}-11k+16)(2k-8)!!$ & $(2k^{2}-11k+16)(2k-8)!!$ & $(2k-8)(2k-8)!!$ & $4(2k^{2}-11k+16)(2k-8)!!$ & $(2k^{2}-9k+12)(2k-7)(2k-8)!!$ \\ [1ex] 
\hline 
\end{tabular}
\vspace*{0.5cm}
\captionof{table}{The adjacency matrix of $X_{[2k-2,2]}/[2k-4,2,2]$}
\end{centering}
\end{sidewaystable}
\begin{sidewaystable}
\centering
\begin{tabular}{N| m{3cm} | m{3cm} | m{3cm} | m{3cm} | m{3cm} | m{3cm} |}
\hline 
\vspace*{0.8cm}
\rule{0pt}{13pt}& $\bf{0}$ & $\bf{0}$ & $\bf{0}$ & $2(2k-6)!!$ & $(2k-4)!!$ & $\frac{1}{4}(2k-2)!!$ \\ [1ex] 
\hline
\vspace*{0.8cm}
\rule{0pt}{13pt}& $\bf{0}$ & $\bf{0}$ & $\frac{1}{2}k(2k-6)!!$ & $\frac{1}{2}(2k-6)!!$ & $(2k-1)(2k-6)!!$ & $\frac{1}{2}(2k^{2}-7k+1)(2k-6)!!$\\[1ex] 
\hline
\vspace*{0.8cm}
\rule{0pt}{13pt}& $\bf{0}$ & $\frac{1}{2}k(2k-6)!!$ & $\bf{0}$ & $\frac{1}{2}(2k-6)!!$ & $(2k-1)(2k-6)!!$ & $\frac{1}{2}(2k^{2}-7k+1)(2k-6)!!$\\[1ex] 
\hline
\vspace*{0.8cm}
\rule{0pt}{13pt}& $(2k-6)!!$ & $\frac{1}{4}(2k-4)!!$ & $\frac{1}{4}(2k-4)!!$ & $(2k-6)!!$ & $\frac{1}{2}(2k-4)!!$ &  $\frac{1}{4}(2k-2)!!$\\[1ex] 
\hline
\vspace*{0.8cm}
\rule{0pt}{13pt}& $\frac{1}{4}(2k-6)!!$ & $\frac{1}{4}(2k-1)(2k-6)!!$ & $\frac{1}{4}(2k-1)(2k-6)!!$ & $\frac{1}{4}(2k-6)!!$ & $\frac{1}{2}(3k-1)(2k-6)!!$ &  $\frac{1}{2}(2k^{2}-7k+1)(2k-6)!!$\\[1ex] 
\hline
\vspace*{0.8cm}
\rule{0pt}{13pt}& $\frac{1}{2}(k-1)(2k-8)!!$ & $\frac{1}{2}(2k^{2}-7k+1)(2k-8)!!$ & $\frac{1}{2}(2k^{2}-7k+1)(2k-8)!!$ & $(k-1)(2k-8)!!$& $2(2k^{2}-7k+1)(2k-8)!!$ & $(2k^{3}-14k^{2}+\frac{51}{2}k-\frac{3}{2})(2k-8)!!$\\[1ex] 
\hline
\end{tabular}
\vspace*{0.5cm}
\caption{The adjacency matrix of $X_{[2k-4,4]}/[2k-4,2,2]$}

\end{sidewaystable}

\end{document}